\documentclass[reqno]{amsart}
\usepackage{amsmath,amssymb,mathrsfs,amsthm,amsfonts}
\usepackage[inline]{enumitem} 
\usepackage[utf8]{inputenc} 
\usepackage[usenames,dvipsnames]{xcolor}
\usepackage{psfrag} 
\usepackage{hyperref}
\hypersetup{%
	colorlinks=true, linkcolor=blue,
	citecolor=ForestGreen
}

\usepackage[paper=letterpaper,margin=1in]{geometry}

\newtheorem{theorem}{Theorem}[section]

\newtheorem{lemma}[theorem]{Lemma}
\newtheorem{proposition}[theorem]{Proposition}
\newtheorem{corollary}[theorem]{Corollary}
\theoremstyle{definition}

\newtheorem{remark}[theorem]{Remark}

\numberwithin{equation}{section}
\allowdisplaybreaks

\usepackage{acronym}
\acrodef{KPZ}{Kardar--Parisi--Zhang}
\acrodef{LD}{Large Deviation}
\acrodef{LDP}{Large Deviation Principle}
\acrodef{PP}{Point Process}
\acrodef{TASEP}{Totally Asymmetric Simple Exclusion Process}
\acrodef{SHE}{Stochastic Heat Equation}
\acrodef{SPDE}{Stochastic Partial Differential Equation}
\acrodef{SAO}{Stochastic Airy Operator}
\acrodef{RCLL}{Right-Continuous-with-Left-Limits}
\acrodef{SDE}{Stochastic Differential Equation}

\renewcommand{\Pr}{ \mathbf{P} }	
\newcommand{\Ex}{ \mathbf{E} }		
\renewcommand{\d}{ \mathrm{d} }		
\newcommand{\sao}{\mathcal{A}}		
\newcommand{\saoo}{\mathcal{A}_*}	
\newcommand{\ao}{\mathsf{A}}		
\newcommand{\hill}{\mathcal{H}}		
\newcommand{\ind}{\mathbf{1}}		
\newcommand{\set}[1]{{\{#1\}}}		
\newcommand{\schro}{S}				
\newcommand{\pot}{J}				

\newcommand{\C}{\mathbb{C}} 
\newcommand{\N}{\mathbb{N}} 
\newcommand{\R}{\mathbb{R}} 
\newcommand{\Z}{\mathbb{Z}} 

\newcommand{\rate}{\Phi_-}			
\newcommand{\cost}{\phi}			

\newcommand{\eigen}{\boldsymbol{\lambda}}	
\newcommand{\img}{\mathbf{i}}				

\newcommand{\intv}{I}			
\newcommand{\T}{\mathbb{T}} 	
\newcommand{\pt}{\eta}			

\newcommand{\Nsao}[1]{N(#1,\sao)} 			
\newcommand{\Nsaoi}[2]{N_i(#1,\sao)}		
\newcommand{\Nhill}[2]{N(#1,\hill_{#2})}	
\newcommand{\Nsaoo}[1]{N(#1,\saoo)} 		
\newcommand{\Nhillt}[2]{N(#1,\widetilde{\hill}_{#2})}	
\newcommand{\Nsaoot}[1]{N(#1,\til{\sao}_*)}	

\newcommand{\hilb}{\mathscr{H}} 	
\newcommand{\hilbb}{\mathscr{V}} 	
\newcommand{\hilbmm}{\mathscr{E}} 	
\newcommand{\ip}[2]{\langle#1,#2\rangle}			
\newcommand{\norm}[1]{\Vert #1\Vert}				

\newcommand{\e}{\varepsilon}
\newcommand{\ic}{\text{ic}}
\newcommand{\hf}{\text{hf}}
\newcommand{\neu}{\text{Neu}}
\newcommand{\calG}{\mathcal{G}}
\newcommand{\calN}{\mathcal{N}}
\newcommand{\calM}{\mathcal{M}}

\newcommand{\til}{\widetilde}
\renewcommand{\bar}{\overline}

\usepackage{graphicx}
\newcommand*{\Cdot}{{\raisebox{-0.5ex}{\scalebox{1.8}{$\cdot$}}}} 
\graphicspath{{./figs/}}

\title[Lower tail LDs of the KPZ equation]
{Exact lower tail large deviations of the KPZ equation}
\author[L.-C.\ Tsai]{Li-Cheng Tsai}
\address{L.-C.\ Tsai,
	Departments of Mathematics, Columbia University,
	\newline\hphantom{\quad \ \ L.-C. Tsai}
	2990 Broadway, New York, NY 10027
	}
\email{lctsai.math@gmail.com}

\begin{document}
\begin{abstract}
Consider the Hopf--Cole solution $ h(t,x) $ of the KPZ equation with narrow wedge initial condition.
Regarding $ t\to\infty $ as a scaling parameter, 
we provide the first rigorous proof of the Large Deviation Principle (LDP) for the lower tail of $ h(2t,0)+\frac{t}{12} $,
with speed $ t^2 $ and an explicit rate function $ \rate(z) $.
This result confirms existing physic predictions~\cite{sasorov17,corwin18a,krajenbrink18}.
Our analysis utilizes the formula from~\cite{borodin16} to convert LDP of the KPZ equation
to calculating an exponential moment of the Airy point process.
To estimate this exponential moment, 
we invoke the stochastic Airy operator,
and use the Riccati transform, 
comparison techniques,
and certain variational characterizations of the relevant functional.
\end{abstract}

\maketitle

\section{Introduction}
\label{intro}
In this article we study the lower tail probability of the \ac{KPZ} equation:
\begin{align*}
	\partial_t h = \tfrac12 \partial_{xx} h + \tfrac12 (\partial_x h)^2 + \xi,
	\qquad
	(t,x) \in [0,\infty)\times\R,
\end{align*}
where $ \xi=\xi(t,x) $ is the spacetime white noise.
Introduced in \cite{kardar86}, the \ac{KPZ} equation is a paradigmatic model for random surface growth,
which has links to a host of different physical phenomena. 
Via the Hopf--Cole transform and the Feynman--Kac formula, 
this equation connects to directed polymer in random environment \cite{huse85}.
The spatial derivative $ \partial_x h $ satisfies the stochastic Burgers equation,
which is a model for randomly stirred fluid \cite{forster77}, interacting particle systems, and driven lattice gases \cite{vanbeijeren85}.
In additional to being a phenomenological model, the \ac{KPZ} equation has been fertile ground for mathematical study.
Being a nonlinear equation and an irreversible Markov process,
\ac{KPZ} equation has been a prototype for the study of \acp{SPDE} 
and 
weakly irreversible interacting particle systems.
Along with a vast host of (discrete and continuous) models, the \ac{KPZ} equation enjoys exact solvability
originating from combinatorics, representation theory, and Bethe ansatz.
We refer to \cite{ferrari10,quastel11,corwin12,quastel15,chandra17} and the references therein.

We say $ h $ is a \textbf{Hopf--Cole solution} of the \ac{KPZ} equation if $ h(t,x)=\log Z(t,x) $,
and the process $ Z(t,x) $ solves the \ac{SHE}
\begin{align}
	\label{eq:she}
	\partial_t Z = \tfrac12 \partial_{xx} Z + \xi Z,
	\qquad
	(t,x) \in [0,\infty)\times\R.
\end{align}
Throughout this article we will consider the \textbf{narrow wedge} initial condition
\begin{align}
	\label{eq:nw}
	Z^\ic(x) = \delta(x).
\end{align}
Such a notion of solution is motivative by informally exponentiating the \ac{KPZ} equation,
and it is the physically relevant notion of solution that has been observed from
various regularization schemes and particle systems, e.g., \cite{bertini95,bertini97}.
Also, for certain class of continuous initial conditions, the Hopf--Cole solution 
agrees with the ones constructed from regularity structures \cite{hairer14}, 
paracontrolled distributions \cite{gubinelli15}, 
and energy solutions \cite{gonccalves14,gubinelli18}.
A slight generalization of the standard theory \cite{walsh86,bertini95} asserts that
there exists a unique $ C((0,\infty),\R) $-valued process $ Z $ that solves \eqref{eq:she}--\eqref{eq:nw} in the mild sense, i.e.,
\begin{align*}
	Z(t,x) = p(t,x) + \int_0^t \int_\R p(t-s,x-y) Z(s,y) \xi(s,y) \, \d s \d y,
\end{align*}
where $ p(t,x) := (2\pi t)^{-\frac12} \exp(-\frac{x^2}{2t}) $ denotes the standard heat kernel.
Further, \cite{mueller91} showed that for almost surely for all $ t>0 $, the solution is strictly positive,
i.e., $ Z(t,x)>0 $, for all $ x\in\R $ and $ t>0 $.
This defines the Hopf--Cole solution $ h(t,x) := \log Z(t,x) $ with the initial condition~\eqref{eq:nw}.

Under the initial condition~\eqref{eq:nw}, for large $ t $,
the height develops an average (downward) growth with velocity $ -\frac{1}{24} $,
and, after centering, fluctuates at $ O(t^\frac13) $ and scales to the GUE Tracy--Widom distribution
\cite{amir11,calabrese10,dotsenko10,sasamoto10}
\begin{align*}
	t^{-\frac13}(h(2t,0) + \tfrac{t}{12}) \Longrightarrow \text{GUE Tracy--Widom,}
	\quad
	\text{as } t\to\infty.
\end{align*}
Here, instead of typical behaviors of $ h $, we focus on \acp{LD}, namely the rare events that $ h(2t,0) $ deviates 
distance $ O(t) $ from its center $ -\frac{t}{12} $.
Regarding $ t\to\infty $ as a scaling parameter, we aim at extracting the leading order 
of the tail probability: 
\begin{align}
	\tag{Upper Tail}
	\Pr\big[ h(2t,0)+\tfrac{t}{12} > z t \big] &\approx \exp\big(-t^{a_+} \Phi_+(z) \big),
	\quad
	z > 0,
\\
	\tag{Lower Tail}
	\Pr\big[ h(2t,0)+\tfrac{t}{12} < z t \big] &\approx \exp\big( -t^{a_-} \rate(z) \big),
	\quad
	z < 0,
\end{align}
as $ t\to\infty $.
We refer to $ t^{a_\pm} $ as the speed of deviations, and $ \Phi_\pm(z) $ as the rate function.

Put in a broader context of random growth, directly polymers, and particle systems,
the upper and lower tail \acp{LD} considered here probe excess growth and die-out, respectively.
Whereas excess growth originates from \emph{locally} favorable environment,
die-out occurs only when a widespread area of environment \emph{jointly} becomes unfavorable.
This distinction results in asymmetric speed: $ t^{a_+}=t^1 $ and $ t^{a_-}=t^2 $, and also manifests itself in the rate function $ \Phi_\pm $.
The upper tail is accessible from Fredholm determinants~\cite[Proposition 10]{corwin13},
and it is predicted \cite{ledoussal16,sasorov17} that $ \Phi_+(z)=\frac43 z^\frac32 $, a single $ \frac32 $-power.
On the other hand, the lower tail rate function is predicted~\cite{kolokolov07,meerson16} to exhibit a crossover from 
cubic power law $ (-z)^3 $ for small $ |z| $ to $ \frac52 $-power law $ (-z)^{\frac52} $ for large $ |z| $.
While the $ \frac32 $-power law is seen also in zero temperature polymer models,
the crossover behavior for lower tail distinguishes \ac{KPZ} equation, as a positive temperature polymer model, from zero temperature polymers.

Given the known Fredholm determinant formula (\cite{amir11,calabrese10,dotsenko10,sasamoto10}, see~\cite[Eq.~(7)]{borodin16}),
extracting the upper tail boils down a perturbative analysis.
This is so because, the relevant operator becomes vanishing (in Hilbert--Schmidt norm) as $ t\to\infty $.
By contrast, for the lower tail, one faces the situation where an operator does not converge to zero yet the determinant does.
This is a well-known issue in random matrix theory, and has since prompted the development for much more involved machineries.
For example, extracting the lower tail of the GUE Tracy--Widom distribution
is done by the method of commuting operators \cite{tracy94}, 
via Riemann--Hilbert problems \cite{baik08},
via the Stochastic Airy Operator \cite{ramirez11},
or non-rigorously via Coulomb gas \cite{dean06}.

The first result regarding lower tail of the \ac{KPZ} equation is 
the aforementioned almost-sure positivity of $ Z $ \cite{mueller91}.
Motivated in part by showing the existence of probability density of $ Z(t,x) $,
there has been works \cite{mueller08,flores14,hu18}
on negative moments and the positivity of $ Z $.
These results mostly concern finite time behaviors of $ Z $,
and, in view of the $ -\frac{t}{12} $ average growth,
are not well-adapted to the $ t\to\infty $ regime.

Recently, there has been much development around accessing the lower tail in the $ t\to\infty $ regime.
In~\cite{corwin18}, \emph{rigorous} upper and lower bounds on the lower tail probability are obtained.
The bounds hold for all sufficiently large $ t $, and capture the aforementioned crossover behavior. 
The upper and lower bounds do not match as $ t\to\infty $, and hence do not yield the rate function $ \rate $.
In physics literature, much machinery has been built toward obtaining the the rate function.
In \cite{sasorov17}, an explicit rate function $ \rate $ (see~\eqref{eq:rate}) was predicted.
This is done by analyzing a generalization of Painlev\'{e} II, introduced in \cite{amir11},
through an infinite-dimensional Riemann--Hilbert problem and a WKB approximation, along with a self-consistency ansatz.
Later, based on a formula from~\cite{borodin16}, 
\cite{corwin18a} employed a Coulomb gas heuristic to derive the rate function $ \rate $, which agrees with result in~\cite{sasorov17}.
More recently, based on certain conjectural forms of expansions, \cite{krajenbrink18} developed a scheme of calculating cumulants under the Airy point process,
and, through resummation, produced the same rate function $ \rate $ previously predicted.

The aforementioned physics results provide much insight in the lower tail \ac{LDP}.
They, however, assume certain conjectural formulas or approximations,
or are based on certain infinite dimensional settings that sit beyond existing theories.
In this work, we give the \emph{first rigorous proof} of the lower tail \ac{LDP} of the KPZ equation, 
by invoking the stochastic Airy operator,
and using the Riccati transform, 
comparison techniques,
and certain variational characterizations.

\begin{theorem}\label{thm:main:}
Let $ h(t,x) $ denote the Hopf--Cole solution of \ac{KPZ} equation with narrow wedge initial condition $ Z(0,x)=\delta(x) $, 
and fix $ \zeta\in(0,\infty) $.
Then
\begin{align*}
	\lim_{t\to\infty} \frac{1}{t^2} \log \big(  \Pr\big[ h(2t,0)+\tfrac{t}{12} < -\zeta t \big] \big)
	=
	-\rate(-\zeta),
\end{align*}
with the rate function
\begin{align}
	\label{eq:rate}
	\rate(z)
	:=
	\tfrac{4}{15\pi^6}(1-\pi^2z)^{\frac52}-\tfrac{4}{15\pi^6}+\tfrac{2}{3\pi^4}z-\tfrac{1}{2\pi^2} z^2,
	\quad
	z \leq 0.
\end{align}
\end{theorem}

The starting point of our analysis is a formula of \cite{borodin16}
that expresses the previously known Fredholm determinant formula \cite{amir11,calabrese10,dotsenko10,sasamoto10} in terms of Airy \ac{PP}.
Even though only the $ \beta=2 $ Airy \ac{PP} will enter the formula,
to demonstrate the generality of our approach, we will consider general $ \beta>0 $.
Let $ B(x) $, $ x \geq 0 $, denote a standard Brownian motion.
Recall from~\cite{ramirez11} that the \ac{SAO}
\begin{align}
\label{eq:sao}
	\sao_\beta := - \frac{\d^2~}{\d x^2} + x + \frac{2}{\sqrt{\beta}} B'(x)
\end{align}
with Dirichlet boundary boundary condition at $ x=0 $ defines a self-adjoint operator on $ L^2(0,\infty) $
(see Section~\ref{sect:tool} for more details on the construction of $ \sao_\beta $).
Further, $ \sao_\beta $ has a pure-point spectrum that is bounded below and has no limit points:
\begin{align*}
	 -\infty \leq \eigen_{1}(\sao_\beta) \leq \eigen_{2}(\sao_\beta) \leq \eigen_{3}(\sao_\beta) \ldots \to\infty.
\end{align*}
The $ \beta $-Airy \ac{PP} $ \{\mathfrak{a}_{k,\beta}\}_{k=1}^\infty $ is simply this spectrum of $ \sao_\beta $ up to a space reversal, i.e.,
$ \mathfrak{a}_{k,\beta} := - \eigen_{k}(\sao_\beta) $.  
In \cite[Theorem~2.1]{borodin16}, substituting $ (\frac{T}{2},\mathfrak{a}_k,u) \mapsto (t,-\eigen_{k}(\sao_2),e^{t\zeta}) $, we have
\begin{align}
\label{eq:BG}
	\Ex\Big[ \exp\Big( -e^{h(2t,0)+\frac{t}{12}+t\zeta} \Big) \Big]
	=
	\Ex\Big[ \exp\Big( -\sum_{k=1}^\infty \cost_{t}\big( \eigen_{k}(\sao_2)-t^{\frac23}\zeta \big) \Big)  \Big],
\end{align}
where
\begin{align}
	\label{eq:costt}
	\cost_{t}(\lambda) := \log\big(1+\exp(-t^{\frac13}\lambda)\big).	
\end{align}
The formula~\eqref{eq:BG} links two \emph{distinct} objects: the \ac{KPZ} equation on the left, and the Airy \ac{PP} process on the right.
These two objects are a priori irrelevant, but specific observables of them match algebraically.

It is readily checked that the double exponential function $ e^{-e^x} $ 
well approximates the indicator function $ \ind_\set{x<0} $ except in a neighborhood of $ x=0 $.
As $ t\to\infty $, it is conceivable that the l.h.s.\ of~\eqref{eq:BG} becomes a good proxy for the tail probability $ \Pr[h(2t,0)+\frac{t}{12} < -\zeta t] $,
and that proving Theorem~\ref{thm:main:} amounts to proving 

\begin{theorem}
\label{thm:main}
For fixed $ \in(0,\infty) $ and $ \zeta,\beta,L\in(0,\infty) $, we have
\begin{align}
	\label{eq:thm:main}
	\lim_{t\to\infty} \frac{1}{t^2} \Ex\Big[ \exp\Big( -L \sum_{k=1}^\infty \cost_{t}\big (\eigen_{k}(\sao_\beta) - t^\frac23\zeta \big) \Big)  \Big]
	=
	-L \, \Big( \frac{2 L}{\beta} \Big)^5 \rate\Big( -\Big(\frac{\beta}{2L}\Big)^2 \zeta \Big).
\end{align}
\end{theorem}

The relevant parameters correspond to the r.h.s.\ of~\eqref{eq:BG} are $ \beta=2 $ and $ L=1 $.
Here, we state and prove Theorem~\ref{thm:main} for general $ \beta,L\in (0,\infty) $ to demonstrate the generality of our method.
Further, it has an application in a different setup.
Referring to \cite[Definition~7.1]{barraquand18},
let $ h^\hf(t,x) := \log Z^\hf(t,x) $ denotes the Hopf--Cole solution of the \ac{KPZ} equation 
on half-line $ [0,\infty) $ with boundary parameter $ A=-\frac12 $, with initial condition $ Z^\hf(0,x) = \delta(x) $.
The result~\cite[Theorem~B]{barraquand18} together with the convergence result of half-space ASEP~\cite{parekh17} (which generalizes the result~\cite{corwin16})
yields the identity
\begin{align}
	\label{eq:bbcw}
	\Ex\Big[ \exp\Big(-\frac14 e^{h^\hf(2t,0)+\frac{t}{12}+t\zeta}\Big) \Big]
	=
	\Ex\Big[ \exp\Big( - \frac12\sum_{k=1}^\infty \cost_t(\eigen_{k}(\sao_1)-t^\frac23\zeta) \Big) \Big].
\end{align}
Indeed, the r.h.s.\ of~\eqref{eq:bbcw} corresponds to $ \beta=1 $ and $ L=\frac12 $.
As a corollary of Theorem~\ref{thm:main} we have
\begin{corollary}
\label{cor:goe}
Referring to \cite[Definition~7.1]{barraquand18},
let $ h^\hf(t,x) := \log Z^\hf(t,x) $ denote the Hopf--Cole solution of the \ac{KPZ} equation 
on half-line $ [0,\infty) $ with boundary parameter $ A=-\frac12 $,
with initial condition $ Z^\hf(0,x) = \delta(x) $.
Then, for any fixed $ \zeta>0 $,
\begin{align*}
	\lim_{t\to\infty} 
	\tfrac{1}{t^2} \log \big(
		\Pr\big[ h^{\hf}(2t,0) + \tfrac{t}{12} < -t\zeta \big]
	\big)
	=
	- \tfrac12 \rate(-\zeta).
\end{align*}
\end{corollary}
\noindent
Passing from Theorem~\ref{thm:main} to Theorem~\ref{thm:main:} and Corollary~\ref{cor:goe} is simple, which we do in Section~\ref{sect:pf:main}.
In addition to Theorem~\ref{thm:main:} and Corollary~\ref{cor:goe},
there may be further connection to the processes considered in \cite{gorin18}, but we do not pursue this direction here.


The preceding discussion reduces the \ac{LDP} of the \ac{KPZ} equation
to calculating an exponential moment of the Airy \ac{PP}.
This observation was first used in~\cite{corwin18}, along with certain bounds on the Airy \ac{PP}, to derive bounds on the lower tail probability.
Further, it was noted \cite[Section~2.3]{corwin18} that the rate function $ \rate $ 
can be derived by developing an \ac{LDP} of the Airy \ac{PP} from the known \ac{LDP} of $ \beta $-ensemble \cite{benarous97}.
This scheme was adopted in \cite{corwin18a}. 
Non-rigorously taking an edge scaling of the known rate function $ I_\beta $ \cite[Theorem 1.3]{benarous97} of $ \beta $-ensemble,
\cite{corwin18a} derived an explicit rate function $ I_\text{Airy} $~\cite[Section~A, Supplementary Material]{corwin18a} for the Airy \ac{PP},
and solved a corresponding variational problem to obtain $ \rate $.

The non-rigorous edge scaling from $ I_\beta $ to $ I_\text{Airy} $ is backed by the known weak convergence~\cite{ramirez11} of the $ \beta $-ensemble to the Airy \ac{PP}.
However, justifying this passage at \ac{LDP} level requires convergence up to \emph{exponentially small} probability,
which remains an open problem.
Here, we proceed through a different approach, and completely bypass the need for taking edge scaling from the $ \beta $-ensemble.

\subsection{A heuristic of the proof}
\label{sect:heu}
We give a heuristic of the ideas behind our proof.
The discussion in this subsection is informal, serves only as a conceptual guideline, and will not be used in the rest of the article.

Let $ \calG_t := L \sum_{k=1}^\infty \cost_{t}\big (\eigen_{k}(\sao_\beta) - t^\frac23\zeta ) $
denote the relevant quantity on the r.h.s.\ of~\eqref{eq:thm:main}.
By Varadhan's lemma, analyzing the $ t\to\infty $ behavior of $ \Ex[\exp(-\calG_t)] $ amounts to characterizing the \acp{LD} of $ \calG_t $.
%
With $ B $ being the only random component in $ \sao_\beta $ (see~\eqref{eq:sao}), the quantity $ \calG_t $ is a functional of $ B $.
Therefore, the \acp{LD} of $ \calG_t $ is ultimately a question on \acp{LD} of a functional of the Brownian motion $ B $.
To better express $ \calG_t $ as a functional of $ B $, we use the Riccati transform. 
Let
\begin{align*}
	N(\lambda) := \# \{k\in\N:\eigen_{k}(\sao_\beta) \leq \lambda\}
\end{align*} 
denote the number of eigenvalues of $ \sao_\beta $ at most $ \lambda $, i.e., counting function, 
and consider the solution of the following ODE
\begin{align}
	\label{eq:riccati:heu}
	f'(x) = x - \lambda - f^2(x) + \tfrac{2}{\sqrt{\beta}} B'(x),
	\
	x>0,
	\quad
	f(0) = +\infty. 
\end{align}
Due to the negative, quadratic drift $ -f^2 $, the solution may undergo a few explosions to $ -\infty $,
whence $ f $ is immediately restarted at $ +\infty $.
the Riccati transform asserts (see Section~\ref{sect:tool} for more details) that $ N(\lambda) = \#\{ $explosions of $ f(x) \} $.
We hence view $ f $ and $ N(\lambda) $ as functionals of $ B $ through~\eqref{eq:riccati:heu},
and this gives $ \calG_t $ as a functional of $ B $ through
\begin{align}
	\label{eq:S:heu}
	\calG_t 
	= 
	L\, \int_{\R} \cost_t(\lambda-t^\frac32\zeta) \d N(\lambda)
	=
	- L\,\int_{\R} \cost'_t(\lambda-t^\frac32\zeta) N(\lambda) \d \lambda.
\end{align}

We now need to analyze how deviations of $ B $ affect $ f $ and $ N(\lambda) $.
To this end, it is instructive to first laid down a few scales.
Straightforward differentiations from~\eqref{eq:costt} shows that 
$ \cost'_t(\lambda-t^\frac23\zeta) \approx -t^\frac13 \ind_\set{\lambda<t^{2/3}\zeta} $ for $ t \gg 1 $.
Using this in~\eqref{eq:S:heu}, we see that the relevant $ \lambda $ should be of order $ t^\frac23 $, i.e., $ \lambda = O(t^\frac23) $.
In~\eqref{eq:riccati:heu}, if we ignore the Brownian term $ \frac{2}{\sqrt{\beta}} B'(x) $, explosions of $ f $ occurs only when $ x \leq \lambda $.
This suggests $ x= O(\lambda)=O(t^\frac23) $.
Now, consider a generic $ v\in C[0,\infty) $.
We postulate that, the relevant deviation is $ B(x) $ behaving like a drifted Brownian motion with drift $ t^\frac23 v(t^{-\frac23} x) $.
Here, the $ (t^{-\frac23}x) $ scaling ensures that the drift varies at scale comparable to $ x=O(t^\frac23) $ in~\eqref{eq:riccati:heu},
and the multiplicative factor $ t^{\frac23} $ guarantees that the drift competes at the same level as $ x-\lambda = O(t^\frac23) $.

We henceforward regard $ v $ as the control function of the \acp{LD} in question.
The \ac{LDP} on sample paths of Brownian motion suggests that
\begin{align*}
	\Pr\big[ B'(x) \approx t^\frac23 v(t^{-\frac23}x) \big] 
	\approx
	\exp\Big( - \int_0^\infty \frac{1}2 t^\frac43 v^2(t^{-\frac23}x) \d x \Big)
	=
	\exp\Big( - t^2 \int_0^\infty \frac{1}2 v^2(x) \d x \Big).
\end{align*}
Indeed, $ B $ is not differentiable, and $ B'(x) \approx t^\frac23 v(t^{-\frac23}x) $ 
merely means that $ t^{-\frac43}B(t^\frac23 x) $ 
approximates $ \int_0^x v(y) \d y $ uniformly in $ x $ over compact subsets.
Here, however, we \emph{informally} equate $ B'(x) $ with $ t^\frac23 v(t^{-\frac23} x) $ in~\eqref{eq:riccati:heu} and write
\begin{align*}
	f'_v(x) = -t^\frac23 \big( -t^{-\frac23} x + t^{-\frac23}\lambda  - \tfrac{2}{\sqrt{\beta}} v(t^{-\frac23} x) \big) - f^2_v(x),
	\
	x>0,
	\quad
	f_v(0) = +\infty. 
\end{align*}
This equation can be solved approximately by regarding $ b(x) := -t^{-\frac23} x + t^{-\frac23}\lambda - 2v(t^{-\frac23} x)/\sqrt{\beta}  $
as a locally constant function.
Consider a generic $ b>0 $ and solve for a function $ f_\text{loc} $ that satisfies $ f'_\text{loc} = - t^\frac23 b - f^2_\text{loc} $. 
This gives $ f_\text{loc}(x) = \tan(t^\frac13 b^{\frac12}x+c) $, $ c\in\R $,
which explodes over a period of $ \pi t^{-\frac13} b^{-\frac12} $.
Hence the time lapse between explosions of $ f_v $ near a given point $ x $ is roughly
\begin{align*}
	\tau_v(x) \approx \pi t^{-\frac13} \big( \big(-t^{-\frac23} x + t^{-\frac23}\lambda - \tfrac{2}{\sqrt{\beta}} v(t^{-\frac23} x) \big)_+ \big)^{-\frac12},
\end{align*} 
where $ y_\pm := (\pm y)\vee 0 $ and $ 1/0 := \infty $.
Integrating the reciprocal time lapse $ 1/\tau_v(x) $ over $ x \geq 0 $ gives the total number of explosions:
\begin{align*}
	N(\lambda) = N_v(\lambda)
	\approx \frac{t^{\frac13}}{\pi} \int_0^\infty \big( \big(-t^{-\frac23} x + t^{-\frac23}\lambda - \tfrac{2}{\sqrt{\beta}} v(t^{-\frac23} x) \big)_+ \big)^{\frac12}
	\d x.
\end{align*} 
Now, substituting this approximate expression of $ N(\lambda) $ in~\eqref{eq:S:heu},
together with the aforementioned approximation $ \cost'_t(\lambda-t^\frac23\zeta) \approx -t^\frac13 \ind_\set{\lambda<t^{2/3}\zeta} $,
we arrive at
\begin{align*}
	\calG_t = \calG_{t,v}
	&\approx
	\frac{t^{\frac23} L}{\pi} 
	\int_{-\infty}^{t^\frac23\zeta} 
	\int_0^\infty \big( (-t^{-\frac23} x + t^{-\frac23}\lambda - \tfrac{2}{\sqrt{\beta}} v(t^{-\frac23} x))_+ \big)^{\frac12}
	\,
	\d \lambda \d x
\\
	&=
	t^{2} \frac{2 L}{3\pi} 
	\int_0^\infty \big( (-x + \zeta - \tfrac{2}{\sqrt{\beta}} v(x))_+ \big)^{\frac32}
	\,
	\d x.
\end{align*}

So far we have derived an approximate expression of $ \calG_t=\calG_{t,v} $ as a functional of the control $ v $,
and the `cost' for realizing a given $ v $ is $ \frac{t^2}2 \int_0^\infty v^2(x) \d x $. 
These discussions suggest
\begin{align*}
	\log \big( \Ex\big[ \calG_t \big] \big)
	\approx 
	- 
	t^2\min_{v} \Big\{
		\int_0^\infty 
		\Big( \frac{2L}{3\pi} \big( (-x + \zeta - \tfrac{2}{\sqrt{\beta}} v(x))_+ \big)^{\frac32} + \frac12 v^2(x) \Big) \d x
	\Big\}.
\end{align*}
The minimizer $ v=v_* $ is solved by straightforward variation, giving
\begin{align*}
	v_*(x) = 4 L^2 \pi^{-2}\beta^{-\frac32} \big( -1 + \sqrt{1+(\tfrac{\beta\pi}{2  L})^2(\zeta-x)_+} \big).
\end{align*}
Substitute in $ v=v_* $. After straightforward but tedious calculations, we get\\
$
	\log ( \Ex[ \calG_t ] )
	\approx 
	- 
	t^2 L (\frac{2L}{\beta})^5 \rate(-(\frac{\beta}{2L})^2\zeta).
$

\subsection{Overview of the proof}
\label{sect:outline}
The crucial assumption behind the preceding heuristic is having locally constant drifts.
That is, we postulate that the `optimal strategy' is achieved by having a drift $ t^\frac23 v(t^{-\frac23} x) $ that is locally constant,
and varies at the macroscopic scale $ O(t^{\frac23}) $.
It is \emph{far from clear} why this is the case.
Indeed, with $ B' $ being rough (not function-valued),
local behaviors of $ B $ at scales $ \lesssim  t^{-\frac13} $ could have dramatic effects on the spectrum of $ \sao_\beta $.

Our proof proceeds through a localization procedure.
That is, we partition $ (0,\infty) $ into intervals of length $ t^\alpha $:
$
	\intv_i := (\pt_{i-1},\pt_{i}],
$
$
	\pt_i := i t^\alpha,
$
and counts the number of explosions of the Riccati ODE within each interval $ \intv_i $.
Our analysis works for \emph{any} fixed exponent $ \alpha\in(-\frac13,\frac23) $.
Note that this range exhausts all mesoscopic scales.
As seen in Section~\ref{sect:heu}, $ t^\frac23 $ is the macroscopic scale of $ x $ and $ \lambda $ in~\eqref{eq:riccati:heu},
while $ t^{-\frac13} $ the microscopic scale of typical time lapse $ \tau_v(x) $ between explosions.

To prove Theorem~\ref{thm:main}, we separately establish upper and lower bounds on the l.h.s.\ of~\eqref{eq:thm:main}.
For the lower bound, within each interval $ \intv_i $, we perform a change-of-measure (via Girsanov's theorem)
so that the Brownian motion has drift $ V_i := t^{\frac23} v_*(t^\frac23 \pt_{i-1}) $.
Within $ \intv_i $, the change in the linear potential $ x $ is negligible, 
and can be well-approximated by the constant $ \pt_{i-1} $.
This being the case, the number of explosions (after the change-of-measure) can be estimated by spectral comparison to the shifted Laplace operator
$
	-\frac{\d^2~}{\d x^2} + \pt_{i-1} + \frac{2}{\sqrt{\beta}} V_i.
$
Doing so eventually yields the desired lower bound.

The harder part of the proof is to obtain a matching upper bound.
This is where we address the aforementioned issue---that the `best strategy' is achieved by a locally constant drift.
More precisely, 
we show the `best strategy' is to have $ B $ constantly drifted within each interval $ \intv_i $.
To this end, we first use $ \cost_{t}(\lambda) \approx -t^\frac13 \lambda_- $
to approximate the relevant quantity as a truncated sum of eigenvalues of certain Hill-type operators (see~\eqref{eq:trunc:sum}).
Next, we show in Proposition~\ref{prop:key}
(after passing to periodic boundary condition as done in Lemma~\ref{lem:periodic})
that the truncated sum is dominated by the one with $ B'(x) $ replaced by its average $ \frac{B(\pt_{i})-B(\pt_{i-1})}{|\intv_i|} $.
Key ingredients behind the proof of Proposition~\ref{prop:key} are the variational characterizations built 
in Lemma~\ref{lem:minimaxx} and \eqref{eq:simple}. 

We note here that most part of our proof works even if $ \cost_{t}(\lambda) $ were replaced by a smooth compactly supported function.
However, the aforementioned variational characterizations (Lemma~\ref{lem:minimaxx} and \eqref{eq:simple}) are tailored to truncated sum of eigenvalues,
and hence apply only for the specific cost function $ \cost_{t}(\lambda) \approx -t^\frac13 \lambda_- $.

\subsection{Quantitative bounds}
In this article, we focus on the $ t\to\infty $ asymptotic of the lower tail probability,
and extract the leading order term, i.e., the rate function $ \rate $.
Our analysis, however, allows much room for more quantitative estimates.
As mentioned in Section~\ref{sect:outline}, the partition can take any size $ t^\alpha $ with $ \alpha\in(-\frac13,\frac23) $.
Optimizing over $ \alpha $ (and a few other parameters within our analysis) should lead to a quantitative estimate
on the tail probability in a similar spirit as \cite{corwin18}. 
We do not pursue this direction here. 

\subsection*{Acknowledgements}
LCT thanks Ivan Corwin, Promit Ghosal, and Pei-Ken Hung for useful discussions,
and thanks Ivan Corwin and Yao-Yuan Mao for comments that improve the presentation of this article.
LCT's research was partially supported by a Junior Fellow award from the Simons Foundation, and by the NSF through DMS-1712575.

\subsection*{Outline}
In Section~\ref{sect:tool}, we prepare a few basic tools.
Based on these tools, in Section~\ref{sect:pf:mainmain:} we prove Theorem~\ref{thm:main}.
In Section~\ref{sect:pf:main}, we settle Theorem~\ref{thm:main:} and Corollary~\ref{cor:goe}.

\section{Basic tools}
\label{sect:tool}
Hereafter throughout the rest of the article, we fix $ L,\zeta,\beta\in(0,\infty) $, and drop dependence on these variables.
For example $ \sao:=\sao_\beta $.

We begin by recalling the classical construction of self-adjoint operators via sesquilinear forms.
Consider Hilbert spaces $ \hilb  $ and $ \hilbb $, both over $ \C $,
equipped with inner products $ \ip{\Cdot}{\Cdot}_\hilb $ and $ \ip{\Cdot}{\Cdot}_\hilbb $
and the thus induced norms $ \norm{\Cdot}_\hilb $ and $ \norm{\Cdot}_\hilbb $, and assume the embedding $ \hilbb \subset \hilb $ as vector spaces.
Consider also a symmetric sesquilinear form $ Q: \hilbb\times\hilbb \to \C $.
The \textbf{associated operator} $ T=(T,D(T)) $ of $ Q $ has domain $ D(T) $ consisting of $ v\in\hilbb $ such that
\begin{align}
	\label{eq:asscOp}
	\exists u\in\hilb \text{ such that } Q(v,v')=\ip{u}{v'}_\hilb, \ \forall v'\in\hilbb,
\end{align}
and, for each $ v\in D(T) $, $ Tv:=u $ is defined to be the (necessarily unique) vector $ u\in\hilb $ that satisfies~\eqref{eq:asscOp};
see \cite[Definition~12.14]{grubb08}.
Recall that $ Q $ is \textbf{coercive} with respect to $ \hilbb\subset\hilb $ if, for some fixed constant $ c<\infty $,
\begin{align*}
	\norm{v}^2_\hilbb \leq c\,(\norm{v}^2_{\hilb}+Q(v,v)),	\quad\forall  v\in\hilbb.
\end{align*}
Recall that $ \hilbb $ is \textbf{compactly embedded} in $ \hilb $ if
$ \norm{v}_{\hilbb} \leq c \norm{v}_{\hilb} $, for some fixed constant $ c<\infty $ and all $ v\in\hilbb $,
and if any $ \norm{\Cdot}_{\hilbb} $-bounded sequence has a $ \norm{\Cdot}_{\hilb} $-convergent subsequence.
It is known (c.f., \cite[Corollary~12.19]{grubb08}) that if $ \hilbb\subset\hilb $ compactly and densely and if $ Q $ is coercive,
then the associated operator $ (T,D(T)) $ is self-adjoint and closed, with $ D(T)\subset \hilbb $ being dense in $ \hilb $.
Furthermore, since $ Q $ is coercive and since $ \hilbb\subset\hilb $ compactly and densely,
$ T $ necessarily has a pure-point spectrum that is bounded below and has no limit points, i.e., $ -\infty < \lambda_1 \leq \lambda_2 \leq \ldots\to\infty $,
with the corresponding eigenvectors forming a complete basis (i.e., dense orthonormal set) of $ \hilb $.
We will call such self-adjoint operators \textbf{standard}.

In the following we will consider quadruples $ (T,Q,\hilbb\subset\hilb) $,
where $ Q $ is a symmetric sesquilinear form on $ \hilbb $ and $ T $ is the associated operator.
The preceding discussion is summarized as follows
\begin{proposition}
\label{prop:op}
Fix a quadruple $ (T,Q,\hilbb\subset\hilb) $ described as in the preceding.
If $ \hilbb\subset\hilb $ compactly and densely, and if $ Q $ is coercive,
then $ T $ is \textbf{standard}:
self-adjoint and has a pure-point spectrum that is bounded below and has no limit points, i.e., $ -\infty < \lambda_1 \leq \lambda_2 \leq \ldots\to\infty $,
with the corresponding eigenvectors forming a complete basis of $ \hilb $.
\end{proposition}

Now, to construct the \ac{SAO}~\eqref{eq:sao}, we let $ \hilb = L^2[0,\infty) $, and 
\begin{align}
	\label{eq:L*}
	\hilbb = L_* := \Big\{ f\in H^1[0,\infty) : f(0)=0, \ \int_0^\infty \big( |f'(x)|^2 + (1+x)|f(x)|^2 \big) \, \d x <\infty \Big\},
\end{align}
equipped with the inner product $ \ip{f}{g}_{L_*} := \int_0^\infty (f'(x) \bar{g}'(x) + (1+x)f(x)\bar{g}(x)) \d x $.
It is standard to check that $ L_* \subset L^2[0,\infty) $ compactly and densely.
Now define the symmetric sesquilinear form $ Q_\text{SAO}: L_*\times L_* \to \C $
\begin{align}
	\label{eq:saoform}
	Q_{\text{SAO}}(f,g) := \int_0^\infty \Big( f'(x) \bar{g}'(x) + \Big(x+ \frac{2}{\sqrt{\beta}}\Big) f(x) \bar{g}(x) B'(x) \Big) \, \d x,
\end{align}
where, with $ f,g\in L_* $, the integral against $ B'(x) $ is understood in the integration-by-parts sense.
Recall from \cite{ramirez11} (see also \cite[Lemma~4.5.44~(b)]{anderson10}) that, almost surely, $ Q_{\text{SAO}} $ is coercive with respect to $ L_*\subset L^2[0,\infty) $.
Given these properties, we let $ \sao $ be the associated operator of $ Q_{\text{SAO}} $, which, by Proposition~\ref{prop:op}, is standard.

Aside from the \ac{SAO}, we will also consider operators of the form
$ -\frac{\d^2~}{\d x^2} + \frac{2}{\sqrt{\beta}} \pot'(x) $, on $ x\in[a,b] $,
for $ \pot \in C[a,b] $, and with Dirichlet boundary condition at $ x=a,b $.
To define such an operator, take $ \hilb = L^2[a,b] $ and $ \hilbb = H^1_0[a,b] := \{f\in H^1[a,b] : f(a)=b(b)=0\} $,
and define
\begin{align}
	\label{eq:schro:form}
	Q_\pot(f,g) := \int_a^b \Big( f'(x) \bar{g}'(x) + \frac{2}{\sqrt{\beta}} f(x) \bar{g}(x) \pot'(x) \Big) \, \d x,
\end{align}
where, for $ f,g\in H^1_0[a,b] $, the integral against $ \pot'(x) $ is understood in the integration-by-parts sense.
Indeed, $ H^1[a,b]\subset L^2[a,b] $ compactly and densely.
For continuous $ \pot $, we show in~\eqref{eq:coar} in the following that $ Q_\pot $ is coercive with respect to $ H^1_0[a,b] \subset L^2[a,b] $.
Given these properties, we let 
\begin{align}
	\label{eq:schro}
	\schro := -\frac{\d^2~}{\d x^2} + \frac{2}{\sqrt{\beta}} \pot'(x),
	\quad
	x\in (a,b),	
	\text{ with Dirichlet BC}
\end{align}
be the operator associated of $ Q_\pot $, which, by Proposition~\ref{prop:op}, is standard.
One particular $ \pot $ we will consider is $ \pot(x)=\frac{2}{\sqrt{\beta}}B(x) $, which gives the Hill operator:
\begin{align}
	\label{eq:hill}
	\hill_{[a,b]} := -\frac{\d^2~}{\d x^2} + \frac{2}{\sqrt{\beta}} B'(x),
	\quad
	x\in (a,b),
	\text{ with Dirichlet BC}.
\end{align}

For a standard operator $ T $, we will often adopt the notation $ \eigen_{k}(T) $ for its $ k $-th eigenvalue,
starting with index $ k=1 $.
For $ (T,Q,\hilbb\subset\hilb) $ satisfying the properties of Proposition~\ref{prop:op}, we have the minimax principle:
\begin{align}
	\label{eq:minimax}
	\eigen_k(T) = \min \Big\{ \max_{v\in\hilbmm, \norm{v}_\hilb=1} \{ Q(v,v) \} : \hilbmm \text{ k-dim.\ subspace of } \hilbb \Big\}.
\end{align}
This principle yields a useful comparison for the spectra of operators of the type~\eqref{eq:schro}.
\begin{lemma}
\label{lem:spec:cmp}
Fix a finite interval $ [a,b] $ and continuous functions $ \pot_i \in C[a,b] $, $ i=1,2 $.
Let $ \schro_i $ be the operators as in~\eqref{eq:schro} with $ \pot_i $ in place of $ \pot $.
We have
\begin{align*}
	\eigen_{n}(\schro_1) 
	\leq
	\tfrac{(\kappa+1)^3}{\kappa^3} \eigen_{n}(\schro_2) 
	+
	\big( \tfrac{(\kappa+1)^2}{\kappa^3} U_2^2 + \kappa^2 U^2_{12} \big),
	\quad
	\kappa >0,
	\
	n=1,2,\ldots,
\end{align*}
where $ U_2:= \sup_{x\in [a,b]} |\pot_2(x)| $ and $ U_{12} := \sup_{x\in [a,b]} |\pot_1(x)-\pot_2(x)| .$
\end{lemma}
\begin{proof}
To simplify notation, we write $ H^1_0 := H^1_0[a,b] $ and $ L^2 := L^2[a,b] $.
For $ \pot\in C[a,b] $, we write $ U_\pot:= \sup_{x\in [a,b]} |\pot(x)| $.
Let $ f\in H^1_0[a,b] $ and $ r>0 $.
Applying the inequality $ 2|a_1a_2| \leq |a_1|^2+|a_2|^2 $ for $ a_1 = r^{\frac12}f'(x) $ and $ a_2=r^{-\frac12} f(x) \pot(x) $, we have
\begin{align*}
	\Big| \int_a^b |f(x)|^2 J'(x) \d x \Big|
	&:=
	\Big| -\int_a^b \big(f(x)\bar{f}'(x)+\bar{f}(x) f'(x)\big) \pot(x) \d x \Big|
\\
	&
	\leq
	\int_a^b \Big( r^{-1}|f'(x)|^2 + r \, |\pot(x)f(x)|^2 \Big) \, \d x
	\leq
	r^{-1} \norm{f'}^2_{L^2} + r \, U_\pot^2 \, \norm{f}^2_{L^2}.
\end{align*}
Setting $ (\pot,r)=(\pot_1-\pot_2,\kappa^2) $ and $ (\pot,r)=(\pot_2,\kappa+1) $ gives
\begin{align}
	\label{eq:formcmp}
	Q_{\pot_1}(f,f)
	\leq
	Q_{\pot_2}(f,f)
	+
	\kappa^{-2} \norm{f'}^2_{L^2} + \kappa^2 U_{12}^2 \norm{f}^2_{L^2},
\end{align}
and 
$  
	Q_{\pot_2}(f,f)
	\geq
	\norm{f'}^2_{L^2} - \frac{1}{\kappa+1} \norm{f'}^2_{L^2} - (\kappa+1) U^2_2 \norm{f}^2_{L^2}.
$
The latter is rearranged as
\begin{align}
	\label{eq:coar}
	\norm{f'}^2_{L^2} &\leq \tfrac{\kappa+1}{\kappa} Q_{\pot_2}(f,f) + \tfrac{(\kappa+1)^2}{\kappa} U_{2}^2 \norm{f}^2_{L^2}.
\end{align}
Inserting~\eqref{eq:coar} into~\eqref{eq:formcmp} gives
\begin{align*}
	Q_{\pot_1}(f,f)
	&\leq
	\big(1+\tfrac{\kappa+1}{\kappa^3}\big) Q_{\pot_2}(f,f) 
	+
	\big( \tfrac{(\kappa+1)^2}{\kappa^3} U_2^2 + \kappa^2 U^2_{12} \big) \norm{f}^2_{L^2}
\\
	&\leq
	\tfrac{(\kappa+1)^3}{\kappa^3} Q_{\pot_2}(f,f) 
	+
	\big( \tfrac{(\kappa+1)^2}{\kappa^3} U_2^2 + \kappa^2 U^2_{12} \big) \norm{f}^2_{L^2}.
\end{align*}
This together with the minimax principle~\eqref{eq:minimax} yields the desired result.
\end{proof}

We will also use the following variational characterization of sums of eigenvalues.
\begin{lemma}
\label{lem:minimaxx}
For $ (T,Q,\hilbb\subset\hilb) $ satisfying the properties of Proposition~\ref{prop:op},
we have
\begin{align*}
	\sum_{k=1}^n \eigen_k(T) = \min \Big\{ \sum_{k=1}^n Q(v_k,v_k) : \{v_1,\ldots,v_n\} \subset \hilbb \text{ orthonormal in } \hilb \Big\}.
\end{align*}
\end{lemma}
\begin{proof}
To simplify notation we write $ \eigen_k(T) = \lambda_{k} $ throughout this proof.
Let $ u_1,u_2,\ldots  $ denote the corresponding orthonormal eigenvectors.
Since $ \lambda_1>-\infty $, by shifting  $ T \mapsto T + c $ and $ Q(v,v') \mapsto Q(v,v')+c\ip{v}{v'}_\hilb $,
we may assume without lost of generality that $ T $ is positive and $ Q $ is elliptic, i.e., $ \norm{v}_{\hilbb}^2 \leq c' Q(v,v) $.
Given any set $ \{v_1,\ldots,v_n\}\subset\hilbb $ that is orthonormal in $ \hilb $, expand each vector into the eigenbasis
$
	v_k = \sum_{i=1}^\infty a^i_k u_i,
$
$
	a^i_k :=  \langle v_k,u_i \rangle_\hilb.
$
Using this we have
\begin{align}
	\label{eq:weight:lambda}
	\sum_{k=1}^n Q(v_k,v_k) 
	=
	\sum_{k=1}^n Q\Big(\sum_{i=1}^\infty a^i_k u_i,\sum_{i'=1}^\infty a^{i'}_k u_{i'}\Big) 	
	=
	\sum_{k=1}^n \sum_{i,i'=1}^\infty a^i_k \bar{a}^{i'}_k Q(u_i,u_{i'}) 
	=
	\sum_{i=1}^\infty\sum_{k=1}^n  |a^i_k|^2 \lambda_i,
\end{align}
where, in the second equality we exchanged infinite sums with $ Q $, which is justified by $ Q $ being elliptic.
Put differently, \eqref{eq:weight:lambda} states that
$ \sum_{k=1}^n Q(v_k,v_k) $ is given by a weighted average of the eigenvalues, with weight $ w_n := \sum_{k=1}^n  |a^i_k|^2 $.
Moreover, the total amount of weight is fixed:
\begin{align*}
	\sum_{n=1}^\infty w_n 
	= 
	\sum_{k=1}^n \sum_{i=1}^\infty |a^i_k|^2
	=
	\sum_{k=1}^n \norm{v_i}^2_{\hilb} = n.
\end{align*}
Given this constraint, to minimize~\eqref{eq:weight:lambda},
it is desirable to allocate more weights to smaller eigenvalues.
On the other hand, each eigenvalue cannot receive weight more than $ 1 $:
\begin{align*}
	w_n 
	= 
	\sum_{k=1}^n \big| \ip{v_n}{u_i}_{\hilb} \big|^2
	\leq
	\norm{u_i}^2_{\hilb} =1,
\end{align*}
where the inequality follows because $ \{v_1,\ldots,v_n\} $ is orthonormal.
Combining the preceding properties, 
we see that the quantity in~\eqref{eq:weight:lambda} cannot be smaller than $ \sum_{k=1}^n\lambda_k $.
Conversely, for $ v_k=u_k $, $ k=1,\ldots,n $, we indeed have $ \sum_{k=1}^n Q(u_k,u_k)=\sum_{k=1}^n\lambda_k $.
\end{proof}

A useful tool for analyzing the eigenvalue distribution is the Riccati transform.
To begin with, the eigenvalue problem for $ \sao $ reads 
\begin{align}
	\label{eq:SAO}
	g''(x) = \tfrac{2}{\sqrt{\beta}} g(x) B'(x) + (x-\eigen) g(x),
	\
	x>0,
\end{align}
understood in the integration-by-parts sense.
Namely, we say $ g\in L_* $ (defined in~\eqref{eq:L*}) solves~\eqref{eq:SAO} 
if it holds upon integrating against any test function $ p(x)\in C^\infty_c[0,\infty) $,
under the interpretation $ \int_0^\infty p(x) g(x) B'(x) \d x := - \int_0^\infty (p'(x) g(x) + p(x)g'(x)) B(x) \d x $. 
The Riccati transform $ f(x) := g'(x)/g(x) $ brings the second order equation~\eqref{eq:SAO} into a first order one
\begin{align*}
	f'(x) = x- \eigen - f^2(x) - \tfrac{2}{\sqrt{\beta}} B'(x),
	\
	x>0.
\end{align*}
More generally, instead of taking an eigenvalue $ \eigen $ of $ \sao $,
we consider a generic $ \lambda \in\R $, regarded as a tunable parameter of the first order equation:
\begin{align}
	\label{eq:riccati}
	f'(x) = x- \lambda - f^2(x) + \tfrac{2}{\sqrt{\beta}} B'(x),
	\
	x>0.
\end{align}
With $ B'(x) $ not being function-valued, we make sense of~\eqref{eq:riccati} by integrating in $ x $.
Note that, due to the negative, quadratic drift $ -f^2(x) $, the solution $ f(x) $ may undergo explosions to $ -\infty $,
so we integrate only over intervals that does not contain such explosions:
\begin{align}
\tag{\ref*{eq:riccati}'}
\label{eq:riccati:int}
\begin{split}
	f(x)\big|_{x_1}^{x_2}
	= 
	\int_{x_1}^{x_2} \big( x-\lambda- f^2(x) \big) \d x 
	&+ 
	\tfrac{2}{\sqrt{\beta}} B(x) \big|_{x_1}^{x_2},
\\
	&
	\notag
	[x_1,x_2] \subset [0,\infty) \text{ such that no explosions occur in } [x_1,x_2].
\end{split}
\end{align}
For a given initial condition $ f_0\in \R $, it is readily checked that
\eqref{eq:riccati:int} permits a unique $ C([0,\tau_1)) $-valued solution $ f $ 
with $ f(0)=f_0 $ until the first explosion time $ \tau_1 $ of $ f $.
We will also consider $ f_0 = +\infty $, which is understood as $ \lim_{x\to 0^+} f(x) = +\infty $.
It is not hard to show that, existence and uniqueness (up to first explosion) holds also for $ f_0=\infty $.
At each explosion $ \tau_n $ to $ -\infty $, we immediately restart $ f $ at $ f(\tau_n)=+\infty $.

Given the prescribed explosion structure, it is convenient to view $ f $ 
as taking value in a countable disjoint union of $ \R $, i.e., 
\begin{align*}
	f\in \R_{-1} \cup \R_{-2} \cup \R_{-3} \cup \ldots := \R_{-\N},
\end{align*}
with each component $ \R_{-i} $ keeping track of the value of $ f $ between the $ (i-1) $-th and $ i $-th explosions.
To define the topology and ordering on $ \R_{-\N} $,
take an order-preserving homeomorphism $ u: \R\to (0,1) $ (e.g., $ u(x) := (\arctan(x)+1)/\pi $),
and consider the map $ \til{u}: \R_{-\N}\to(0,\infty) $: $ \til{u}(x,n):= u(x)-n-1 $.
That is, each $ \R_{-i} $ is mapped into $ (n-1,n) $ in an order-preserving and homeomorphic manner.
We endow the space $ \R^*_{-\N} $ with the pull-back topology and ordering through $ \til{u} $.
Indeed, the latter is simply lexicographical ordering, i.e., $ (x,n) > (x',n') $ if $ n>n' \in -\N $, and $ (x,n) \geq (x',n) $ if $ x \geq x'\in\R $.

We now recall known properties on the Riccati transform that will be used subsequently.
Hereafter, for a standard operator $ T $, we let $ N(\lambda,T) $ denote the counting function of eigenvalues:
\begin{align*}
	N(\lambda,T) = \# \big\{ n\in\N : \eigen_{n}(T) \leq \lambda \big\}.
\end{align*}

\begin{proposition}[\cite{ramirez11}] 
\label{prop:riccati}
Under the prescribed ordering and topology,
\begin{enumerate}[label=(\alph*),leftmargin=7ex]
\item \label{prop:riccati:ode} 
	Fix $ \lambda\in\R $ and an initial condition $ f(0) \in \R\cup\{+\infty\} $,
	equation~\eqref{eq:riccati}--\eqref{eq:riccati:int} admits a unique, continuous solution $ f(x)=f(x,\lambda) $.
	Further, $ f(x,\lambda) $ is decreasing in $ \lambda $ for each $ x $.
\item \label{prop:riccati:order}
	Equation~\eqref{eq:riccati}--\eqref{eq:riccati:int} preserves ordering.
	That is, given any continuous solutions $ f_1(x) $ and $ f_2(x) $ of~\eqref{eq:riccati} with $ f_1(0) \geq f_2(0) $,
	we have $ f_1(x) \geq f_2(x) $ for all $ x \geq 0 $.
\item \label{prop:riccati:num}
	Almost surely for all $ \lambda $,
	$
		\Nsao{\lambda} = \#\{ \text{explosions of } f(\Cdot,\lambda) \text{ in } (0,\infty) \}.
	$
\end{enumerate}
\end{proposition}
\noindent
Parts~\ref{prop:riccati:ode} and \ref{prop:riccati:num} are stated in \cite[Fact~3.1, Proposition~3.5]{ramirez11},
and Part~\ref{prop:riccati:order} follows immediately from Part~\ref{prop:riccati:ode}.
Let us emphasize that, our discussions regarding Riccati transform is \emph{pathwise},
and in particular hold if $ B $ is replaced by any $ w \in C([0,\infty)) $ 
with sublinear growth: $ \lim_{x\to\infty}|g(x)| x^{-a}=0 $, for some $ a<1 $.

As for the Hill operator, similarly consider the Riccati transform:
\begin{align}
	\label{eq:Riccati}
	f'(x) = - \lambda - f^2(x) + \tfrac{2}{\sqrt{\beta}} B'(x),
	\
	x\in(a,b).
\end{align}
Just like in the preceding, we interpret~\eqref{eq:Riccati} in the integrated sense
\begin{align}
\tag{\ref*{eq:Riccati}'}
\label{eq:Riccati:int}
\begin{split}
	f(x)\big|_{x_1}^{x_2}
	= 
	\int_{x_1}^{x_2} \big( -\lambda - f^2(x) \big) \d x 
	&+ 
	\tfrac{2}{\sqrt{\beta}} B(x) \big|_{x_1}^{x_2},
\\
	&
	\notag
	[x_1,x_2] \subset [a,b] \text{ such that no explosions occur in } [x_1,x_2],
\end{split}
\end{align}
and whenever an explosion occurs $ f $ is immediately restarted at $ +\infty $.
It is standard to show (see \cite{fukushima77}) that the following analog of Proposition~\ref{prop:riccati} holds 
\begin{proposition}
\label{prop:Riccati}
Under the prescribed ordering and topology,
\begin{enumerate}[label=(\alph*),leftmargin=7ex]
\item \label{prop:Riccati:ode} 
	Fix $ \lambda\in\R $ and an initial condition $ f(0) \in \R\cup\{+\infty\} $,
	equation~\eqref{eq:Riccati}--\eqref{eq:Riccati:int} admits a unique, continuous solution $ f(x)=f(x,\lambda) $.
	Further, $ f(x,\lambda) $ is decreasing in $ \lambda $ for each $ x $.
\item \label{prop:Riccati:order}
	Equation~\eqref{eq:Riccati}--\eqref{eq:Riccati:int} preserves ordering.
	That is, given any continuous solutions $ f_1(x) $ and $ f_2(x) $ of~\eqref{eq:riccati} with $ f_1(0) \geq f_2(0) $,
	we have $ f_1(x) \geq f_2(x) $ for all $ x \geq 0 $.
\item \label{prop:Riccati:num}
	Almost surely for all $ \lambda $,
	$
		\Nhill{\lambda}{a,b} = \#\{ \text{explosions of } f(\Cdot,\lambda) \text{ in } (a,b] \}.
	$
\end{enumerate}
\end{proposition}

As mentioned previously in Section~\ref{sect:outline}, our proof of Theorem~\ref{thm:main} proceeds by a localization procedure.
To setup notation for it,
fix $ \alpha\in(-\frac13,\frac23) $, and partition $ (0,\infty) $ into intervals of length $ t^\alpha $ up to just beyond the point $ \zeta t^{\frac23} $.
That is, we set $ i_* := \lceil \zeta t^{\frac23-\alpha} \rceil +1 $, $ \pt_i := i t^{\alpha} $, and
\begin{align}
	\label{eq:intv}
	\intv_i := (\pt_{i-1},\pt_{i}],
	\
	i=1,2, \ldots i_*,
	\quad
	\intv_{i*+1} := [\pt_i,\infty).
\end{align}
Accordingly, we count the number of explosions of~\eqref{eq:riccati} on each subinterval
\begin{align*}
	\Nsaoi{\lambda}{i} := \# \{ x\in \intv_i : \lim_{y\to x^-} f(y,\lambda)=-\infty \},
\end{align*}
where $ f(x,\lambda) $ solves~\eqref{eq:riccati} with the initial condition $ f(0,\lambda)=+\infty $.
Then,
\begin{align}
	\label{eq:Nsao:decmp}
	\Nsao{\lambda} = \sum_{i=1}^{i_*+1} \Nsaoi{\lambda}{i}.
\end{align}
Note that we have omitted the dependence on $ t $ in the notation $ \intv_i $, $ \pt_{i} $, etc.
Similar convention will be frequently adopted without explicitly stating.

Indeed, $ \Nsaoi{\lambda}{i} $ depends on the entrance value $ f(\pt_{i-1},\lambda) $ of $ f $ at the start $ \pt_{i-1} $ of the interval $ \intv_i $.
As a result the processes $ \Nsaoi{\Cdot}{i} $, $ i=1,\ldots,i_*+1 $ are mutually dependent.
This being the case, it will often be more convenient to consider
\begin{align*}
	\Nhill{\lambda}{\intv_i},
	\
	i=1,\ldots,i_*,
	\quad
	\Nsaoo{\lambda} := \# \big\{ k\in\N : \eigen_k(\saoo) \leq \lambda \big\},
\end{align*}
where $ \saoo $ is the \ac{SAO} restricted to $ [\pt_{i_*},\infty) $:
\begin{align}
	\label{eq:saoo}
	\saoo := - \frac{\d^2~}{\d x^2} + x + \frac{2}{\sqrt{\beta}} B'(x),
	\quad
	x \geq \pt_{i_*},
	\text{ with Dirichlet BC at } x=\pt_{i_*},
\end{align}
constructed in a similar way as the \ac{SAO}.
Recall from Proposition~\ref{prop:Riccati}\ref{prop:Riccati:num} that
$ \Nhill{\lambda}{\intv_i} $ counts the number of explosions within $ x\in \intv_{i} $ of the solution
$ f_i(x) =f_i(x,\lambda) $ of
\begin{align}
	\label{eq:riccati:i}
	f_i'(x) = - \lambda - f_i^2(x) + \tfrac{2}{\sqrt{\beta}} B'(x),
	\
	x \in \intv_i,
	\quad
	f_i(\pt_{i-1})=+\infty.
\end{align}
Similarly,
$ \Nsaoo{\lambda} $ counts the number of explosions within $ x\in \intv_{i_*+1} $ of the solution
$ f_*(x) =f_*(x,\lambda) $ of
\begin{align}
	\tag{\ref*{eq:riccati}*}
	\label{eq:riccati*}
	f_*'(x) = x- \lambda - f_*^2(x) + \tfrac{2}{\sqrt{\beta}} B'(x),
	\
	x\in\intv_{i_*+1},
	\quad
	f_*(\pt_{i_*})=+\infty.
\end{align}
From the preceding descriptions,
we see that $ \Nhill{\lambda}{\intv_i} $ depends only on the increment $ B(x)-B(\pt_{i-1}) $ of the Brownian motion within $ x\in \intv_i $,
and $ \Nsaoo{\lambda} $ depends only on $ B(x)-B(\pt_{i-1}) $ for $ x\in \intv_{i_*+1} $.
Hence, the processes $ \Nhill{\Cdot}{\intv_i} $, $ i=1,\ldots,i_* $, and $ \Nsaoo{\Cdot} $ are independent.

To relate the processes $ \Nhill{\Cdot}{\intv_i} $ and $ \Nsaoo{\Cdot} $ back to $ \Nsaoi{\Cdot}{i} $,
we establish the following inequalities. 
\begin{lemma}
\label{lem:loc}
Couple the processes $ \Nsaoi{\Cdot}{i} $, $ \Nhill{\Cdot}{\intv_i} $, $ \Nsaoo{\Cdot} $
by having the same spatial white noise $ B'(x) $ for the operators in~\eqref{eq:sao}, \eqref{eq:hill}, and \eqref{eq:saoo}.
Almost surely for all $ \lambda\in\R $ and $ i=1,\ldots,i_* $, we have
\begin{align*}
	\Nhill{\lambda-\pt_{i}}{\intv_i} \leq \Nsaoi{\lambda}{i} \leq \Nhill{\lambda-\pt_{i-1}}{\intv_i }+ 1,
	\quad
	\Nsaoi{\lambda}{i_*+1} \leq \Nsaoo{\lambda} + 1.
\end{align*}
\end{lemma}
\begin{proof}
Fix $ i $ and $ \lambda $.
Let $ f(x) = f(x,\lambda) $ be the solution of~\eqref{eq:riccati} with $ f(0)=+\infty $. 
Restricting~\eqref{eq:riccati} to the relevant interval $ x\in\intv_i $, we write
\begin{align}
	\label{eq:cmp1}
	f'(x) = x- \lambda - f^2(x) + \tfrac{2}{\sqrt{\beta}} B'(x),
	\
	x \in \intv_i,
	\quad
	f(\pt_{i-1}) \in \R\cup\{+\infty\}, \text{ given}.
\end{align}
Let $ g(x)=f_i(x,\lambda-\pt_{i}) $ be the solution of~\eqref{eq:riccati:i} with $ \lambda\mapsto \lambda-\pt_{i} $, i.e.,
\begin{align}
	\label{eq:cmp2}
	g'(x) = -(\lambda-\pt_{i}) - g^2(x) + \tfrac{2}{\sqrt{\beta}} B'(x),
	\
	x\in I_i,
	\quad
	g(\pt_{i-1}) = +\infty.
\end{align}
By definition, $ \Nsaoi{\lambda}{i} $ is the number of explosions of $ f $ on $ \intv_i= (\pt_{i-1},\pt_{i}] $,
and recall that $ \Nhill{\lambda-\pt_{i}}{\intv_i} $ is equal to the number of explosions of $ g $ in $ \intv_i $.
Since $ x-\lambda \leq -(\lambda-\pt_{i}) $ on $ x\in \intv_i $
and since $ f(\pt_{i-1})\leq g(\pt_{i-1})=+\infty $, by comparison we have $ f(x) \leq g(x) $, $ x\in \intv_i $, under the ordering of $ \R_{-\N} $.
This gives the first inequality $ \Nhill{\lambda-\pt_{i}}{\intv_i} \leq \Nsaoi{\lambda}{i} $.

Turning to the second inequality, we consider $ \til{g}(x) = f_i(x,\lambda-\pt_{i-1}) $, which solves
\begin{align}
	\label{eq:cmp3}
	\til{g}'(x) = - (\lambda-\pt_{i-1}) - \til{g}^2(x) + \tfrac{2}{\sqrt{\beta}} B'(x),
	\
	x\in I_i,
	\quad
	\til{g}(\pt_{i-1}) = +\infty,
\end{align}
and consider the first explosion time of $ f $ on $ \intv_i $.
If $ f $ does not explode within $ \intv_i $, then $ \Nsaoi{\lambda}{i} =0 $,
whence the desired inequality $ \Nsaoi{\lambda}{i} \leq \Nhill{\lambda-\pt_{i-1}}{\intv_i} + 1 $ follows trivially.
Otherwise let $ b\in[\pt_{i-1},\pt_{i}] $ denote the first explosion.
We then have $ x-\lambda \geq -(\lambda-\pt_{i}) $ on $ x\in [b,\pt_{i}] $ and $ +\infty=f(b)\geq \til{g}(b) $.
Comparison applied to $ f $ and $ \til{g} $ over the interval $ x\in[b,\pt_{i}] $ yields $ f(x) \geq \til{g}(x) $, $ x\in [b,\pt_{i}] $.
Taking into account the explosion of $ f $ at $ x=b $,
we obtain $ \Nsaoi{\lambda}{i} \leq \Nhill{\lambda-\pt_{i-1}}{\intv_i} + 1 $.

The last inequality concerning $ \Nsaoi{\lambda}{i_*+1} $ and $ \Nsaoo{\lambda} $
follows by the same comparison argument applied to solutions of \eqref{eq:riccati*} and~\eqref{eq:cmp1} for $ i=i_{*}+1 $.
\end{proof}

\section{Proof of Theorem~\ref{thm:main}}
\label{sect:pf:mainmain:}

Our proof of Theorem~\ref{thm:main} breaks into lower and upper bounds.
That is, we establish matching bounds on the l.h.s.\ of~\eqref{eq:thm:main} to obtain the desired result.
Hereafter, we use $ c=(a,b,\ldots) $ to denote a generic, deterministic, finite positive constant that may change from line to line,
but depend only on the designated variables.
As declared previously, $ \beta,\zeta,L\in(0,\infty) $ are fixed throughout this article, 
so their dependence will not be designated.
\subsection{Lower bound}
\label{sect:pflw}
To simplify notation, set 
\begin{align}
	\label{eq:goal}
	G := \Ex \Big[ \exp\Big( - L \, \sum_{k=1}^\infty \cost_{t}(\eigen_{k}(\sao)-t^\frac23\zeta) \Big) \Big].
\end{align}
Our goal is to establish a desired lower bound on $ t^{-2} \log G $.
The proof is carried out in steps.

\medskip
\noindent\textbf{Step 1: localization.}
Recall the partition~\eqref{eq:intv} introduced previously.
By Proposition~\ref{prop:riccati}\ref{prop:riccati:num}, $ \Nsao{\lambda} $ counts the number of eigenvalues $ \eigen_{k}(\sao) $ of the $ \sao $ at most $ \lambda $.
Using this interpretation, together with the decomposition \eqref{eq:Nsao:decmp}, we rewrite the infinite sum in~\eqref{eq:goal} as
\begin{align}
\label{eq:lemloc:1}
\begin{split}
	-\sum_{k=1}^\infty \cost_{t}(\eigen_{k}(\sao)-t^\frac23\zeta)
	&=
	-\int_{\R} \cost_{t}(\lambda-t^\frac23\zeta) \, \d \Nsao{\lambda}
\\
	&=
	\int_{\R} \Nsao{\lambda} \cost_{t}'(\lambda-t^\frac23\zeta) \, \d\lambda
	=
	\sum_{i=1}^{i_*+1} \int_{\R} \Nsaoi{\lambda+t^\frac23\zeta}{i} \cost_{t}'(\lambda) \, \d\lambda,
\end{split}
\end{align}
where $ \d $ acts on the variable $ \lambda\in\R $.
Recall the Hill operator $ \hill_{\intv_i} $ from~\eqref{eq:hill} and $ \saoo $ from~\eqref{eq:saoo}.
Our goal here is to pass from the operator $ \sao $ to $ \hill_{\intv_i} $ for $ i=1,\ldots,i_* $ and to $ \saoo $ for $ i=i_*+1 $.
To simplify notation set $ \calN_i(\lambda) := \Nhill{\lambda+t^\frac23\zeta-\pt_{i-1}}{\intv_i} $ for $ i=1,\ldots,i_* $,
and $ \calN_{i_*+1} := \Nsaoo{\lambda+t^\frac23\zeta} $.
Consider the event $ \Omega_1 := \{ \eigen_{1}(\sao) > -t^{\frac23} \} $ that the groundstate eigenvalue of $ \sao $ lies above $ -t^{\frac23} $.
It is readily checked from~\eqref{eq:costt} that $ \cost_t'(\lambda) < 0 $.
Using this and the bounds from Lemma~\ref{lem:loc} in \eqref{eq:lemloc:1}, we write
\begin{align}
	\notag
	G
	&\geq
	\Ex\Big[
		\ind_{\Omega_1} 
		\cdot 
		\prod_{i=1}^{i_*+1} 
		\exp\Big( L \,\int_{-t^\frac23(1+\zeta)}^\infty \Nsaoi{\lambda+t^\frac23\zeta}{i} \cost_{t}'(\lambda) \, \d\lambda \Big)
	\Big]
\\
	\label{eq:lwbd:0.1}
	&\geq
	\Ex\Big[
		\ind_{\Omega_1} 
		\cdot 
		\prod_{i=1}^{i_*+1} \exp\Big( L \,\int_{-t^\frac23(1+\zeta)}^\infty (1+\calN_i(\lambda)) \cost_{t}'(\lambda) \, \d\lambda \Big)
	\Big].
\end{align}
Within the last expression,
separate the $ 1 $'s from the $ \calN_i $'s and evaluate the contribution of the former
\begin{align*}
	L \, \int_{-t^\frac23(1+\zeta)}^\infty 1\cdot \cost_t'(\lambda)\d \lambda 
	= 
	- L \,\cost_t(t^\frac23(1+\zeta)) 
	= 
	-L \, t^\frac13 \log(1+e^{t^\frac23(1+\zeta)}) \geq -c t.
\end{align*}
With $ i_*+1 \leq c t^{\frac23-\alpha} $, we bound
\begin{align*}
	\prod_{i=1}^{i_*+1}\exp\Big( L \, \int_{-t^\frac23(1+\zeta)}^\infty 1\cdot\cost_{t}'(\lambda) \, \d\lambda \Big)
	\geq 
	e^{-ct^{\frac53-\alpha}}.
\end{align*}
Use this bound in~\eqref{eq:lwbd:0.1},
and then release the remain integral of $ \calN_i\cdot \cost_t' $ (which is negative) to $ \lambda\in\R $ to get
\begin{align}
	\label{eq:lwbd:1}
	G
	\geq
	e^{-ct^{\frac53-\alpha}}
	\Ex\Big[
		\ind_{\Omega_1} 
		\cdot 
		\prod_{i=1}^{i_*+1} \exp\Big( L \, \int_{\R} \calN_i(\lambda) \cost_{t}'(\lambda) \, \d\lambda \Big)
	\Big].
\end{align}

\medskip
\noindent\textbf{Step 2: change of measure.}
Write $ y_\pm := (\pm y)\vee 0 $ for the positive/negative part, and consider
\begin{align}
	\label{eq:v}
	v_*(x) := 4 L^2 \pi^{-2}\beta^{-\frac32} \, \Big( -1 + \sqrt{1+\big(\tfrac{\pi\beta}{2L}\big)^2(\zeta-x)_+} \Big),
\end{align}
and set
\begin{align}
	\label{eq:V}
	V_i := t^\frac23 v_*(t^{-\frac23}\pt_{i-1}),
	\quad
	V(x) := \sum_{i=1}^{i_*} V_i \ind_{\intv_i}(x).
\end{align}
Girsanov's theorem asserts that 
\begin{align}
	\label{eq:girsanov}
	\Ex[\,\Cdot\,] = \til{\Ex}\big[ e^{-\int_0^\infty V(x) \d B(x) +\frac12\int_0^\infty V^2(x) \d x}\,(\,\Cdot\,) \big],
\end{align}
and, under $ \til{\Ex} $, $ B $ is distributed as a drifted Brownian motion,
i.e., $  B \stackrel{\text{law}}{=} \til{B} + \int_0^\Cdot V(y)\d y $, where $ \til{B} $ is a standard Brownian motion.
Let $ \til{\sao}_* = -\frac{\d^2~}{\d x^2} + x + \frac{2}{\sqrt{\beta}} \til{B}'(x) $, $ x \geq \pt_{i_*} $,
and $ \til{\hill}_{\intv_i} = -\frac{\d^2~}{\d x^2} + \frac{2}{\sqrt{\beta}} \til{B}'(x) $, $ x\in\intv_i $ denote the analogous operators.
One the r.h.s.\ of~\eqref{eq:lwbd:1}, apply~\eqref{eq:girsanov}, and express each $ B $ in terms of $ \til{B} $ and $ V $ for the result.
We obtain
\begin{align}
	\label{eq:lwbd:2}
	G
	\geq
	e^{-ct^{\frac53-\alpha}-\frac12 \int_0^\infty V^2(x) \d x }
	\cdot
	\til{\Ex}\Big[
		\ind_{\til\Omega_2} 
		e^{-\int_0^\infty V(x) \d \til{B}(x)}
		\cdot 
		\prod_{i=1}^{i_*+1} \exp\Big( L \, \int_{\R} \til{\calN}_i(\lambda) \cost_{t}'(\lambda) \, \d\lambda \Big)
	\Big],
\end{align}
where
\begin{align*}
	\til{\calN}_i(\lambda) := \Nhillt{\lambda+t^\frac23\zeta-\tfrac{2}{\sqrt{\beta}}V_i}{\intv_i},
	\ 
	i=1,\ldots,i_*,
	\quad
	\til{\calN}_{i_*+1}(\lambda) := \Nsaoot{\lambda+t^\frac23\zeta},
\end{align*}
and $ \til\Omega_2 := \{ \eigen_1(\til\sao+\frac{2}{\sqrt{\beta}}V) > -t^{-\frac23} \} $.
In the last expression we interpreted $ V $ as a multiplicative operator $ L^2[0,\infty)\to L^2[0,\infty) $, which is a bounded, Hermitian operator.
From this point onward, we will always operate under the transformed measure $ \til{\Ex} $.
To alleviate heavy notation, we dropped all the tildes and rewrite~\eqref{eq:lwbd:2} as
\begin{align}
	\tag{\ref*{eq:lwbd:2}'}
	\label{eq:lwbd:3}
	G
	\geq
	e^{-ct^{\frac53-\alpha}-\frac12 \int_0^\infty V^2(x) \d x }
	\cdot
	\Ex \Big[
		\ind_{\Omega_2} 
		e^{-\int_0^\infty V(x) \d B(x)}
		\cdot 
		\prod_{i=1}^{i_*+1} \exp\Big( L \,\int_{\R} \calM_i(\lambda) \cost_{t}'(\lambda) \, \d\lambda \Big)
	\Big].
\end{align}
where $ \Omega_2 := \{ \eigen_1(\sao+\frac{2}{\sqrt{\beta}}V) > -t^{-\frac23} \} $, and
\begin{align*}
	\calM_i(\lambda) := \Nhill{\lambda+t^\frac23\zeta-\tfrac{2}{\sqrt{\beta}}V_i-\pt_{i-1}}{\intv_i},
	\
	i=1,\ldots,i_*,
	\quad
	\calM_{i_*+1}(\lambda) := \Nsaoo{\lambda+t^\frac23\zeta}.
\end{align*}

\medskip
\noindent\textbf{Step 3: bounding terms on the r.h.s.\ of~\eqref{eq:lwbd:3}.}
We begin with the term $ \calM_i(\lambda) $, $ i=1,\ldots,i_* $.
To bound $ \calM_i(\lambda) $, we will apply spectral comparison of the Hill operator $ \hill_{\intv_i} $ and the Laplace operator
\begin{align*}
	-\Delta_{\intv_i} := -\frac{\d^2~}{\d x^2},
	\quad
	\text{with Dirichlet BC}.
\end{align*}
Set $ U_i := \max_{x\in\intv_i} |B(x)-B(\pt_{i-1})| $, fix $ i=1,\ldots,i_* $,
and let $ \kappa \geq 1 $ be an auxiliary parameter. 
Apply Lemma~\ref{lem:spec:cmp} with $ (\pot_1(x),\pot_2(x))=(0,\frac{2}{\sqrt{\beta}}(B(x)-B(\pt_{i-1}))) $ to get
\begin{align}
	\eigen_{n}(\hill_{\intv_i}) \geq \frac{\kappa^3}{(\kappa+1)^3} \eigen_{n}(-\Delta_{\intv_i}) -  c\,(\kappa+1)^2 U_i^2.
\end{align}
From this we deduce, for $ r=t^\frac23\zeta-\tfrac{2}{\sqrt{\beta}}V_i-\pt_{i-1} $,
\begin{align*}
	\calM_i(\lambda) = \# \big\{ n\in\N : \eigen_{n}(\hill_{\intv_i}) \leq \lambda+ r \big\}
	&\leq
	\# \big\{ n\in\N : (\tfrac{\kappa}{\kappa+1})^3 \eigen_{n}(-\Delta_{\intv_i}) \leq \lambda+r + c\,(\kappa+1)^2 U_i^2 \big\}
\\
	&= N\big( (\tfrac{\kappa+1}{\kappa})^3(\lambda+r) + (\kappa+1)^2 c_\star U_i^2 , -\Delta_{\intv_i} \big),
\end{align*}
for some fixed constant $ c_\star<\infty $.
Fix $ \delta \in (0,\frac{2}{3}-\alpha) $ and consider the event
\begin{subequations}
\begin{align}
	\label{eq:Omega3:a}
	\Omega_3(\kappa) := \big\{ (\kappa+1)^2 c_\star U_i^2 &\leq t^{\delta+\alpha_+},
\\
	\label{eq:Omega3:b}
	U_i &\leq t^{\frac{1}{2}(\delta+\alpha_+)}, \ i=1,\ldots,i_* \, \big\}.
\end{align}
\end{subequations}
Given that the interval $ \intv_i $ has length $ |\intv_i| = t^{\alpha} $, it is straightforward to verify
$ \Pr[\Omega_3(\kappa)] \to 1 $, for fixed $ \kappa \in (0,\infty) $ as $ t\to\infty $.
Under the condition~\eqref{eq:Omega3:a}, we have
\begin{align}
	\label{eq:calMibd}
	\ind_{\Omega_3(\kappa)}
	\calM_i(\lambda) 
	\leq 
	M_i(\lambda,\kappa),
	\quad
	i=1,\ldots,i_*,
\end{align}
where
\begin{align}
	\label{eq:calMi}
	M_i(\lambda,\kappa)
	:= 
	N\big( (\tfrac{\kappa+1}{\kappa})^3(\lambda+r_i)+t^{\delta+\alpha_+} , -\Delta_{\intv_i} \big),
	\quad
	r_i:=t^\frac23\zeta-\tfrac{2}{\sqrt{\beta}}V_i-\pt_{i-1}.
\end{align}

We now turn to bounding $ \calM_{i_*+1}(\lambda)=\Nsaoo{\lambda+t^\frac23\zeta} $.
Shifting the operator $ \saoo $ (defined in~\eqref{eq:saoo}) by $ x\mapsto x- \pt_{i_*} $, we see that
$
	\{ \eigen_{n}(\saoo) \}_{n=1}^\infty
	\stackrel{\text{law}}{=}
	\{ \eigen_{n}(\sao) + \pt_{i_*} \}_{n=1}^\infty,	
$
or equivalently
\begin{align}
	\label{eq:Nsao:law}
	\calM_{i_*+1}(\Cdot)
	\stackrel{\text{law}}{=}
	\Nsao{\Cdot+t^\frac23\zeta-\pt_{i_*}}.
\end{align}
Our next step is to compare the spectrum of $ \hill $ to that of the Airy operator $ \ao := - \frac{\d^2~}{\d x^2} + x $,
in a way similarly to Lemma~\ref{lem:spec:cmp}.
Recall that $ \sao $ is the associated operator of the form~\eqref{eq:saoform},
with $ \hilbb = L_* $ given in~\eqref{eq:L*} and $ \hilb = L^2[0,\infty) $.
For the Airy operator,
we take the same Hilbert spaces $ \hilbb = L_{*} \subset \hilb= L^2[0,\infty) $,
with the form $ Q_{\ao}(f,g) := \int_{0}^\infty (f'(x) \bar{g}'(x) + x f(x)\bar{g}(x)) \d x $. 
By \cite[Lemma~4.5.44~(b)]{anderson10}, there exists a $ [0,\infty) $-valued random variables $ U $ such that,
\begin{align*}
	Q_{\sao}(f,f) \geq  \tfrac12 Q_{\ao}(f,f) -  U \norm{f}^2_{L^2[0,\infty)},
	\quad
	\forall f \in L_*.
\end{align*}
The minimax principle~\eqref{eq:minimax} hence gives
$ \eigen_{n}(\sao) \geq \frac12 \eigen_{n}(\ao) -  U. $
From this we conclude
$
	\Nsao{\lambda+t^\frac23\zeta-\pt_{i_*}}
	\leq
	N( 2(\lambda+t^{\frac23}\zeta-\pt_{i_*})+2U, \ao ).
$
Given that $ \pt_{i_*} = i_* t^{\alpha} = (\lceil \zeta t^{\frac23-\alpha} \rceil +1) t^{\alpha} \geq t^\frac23\zeta + t^\frac23 $,
we further obtain
\begin{align}
	\label{eq:Nsao:law:}
	\Nsao{\lambda+t^\frac23\zeta-\pt_{i_*}}
	\leq
	N( 2(\lambda -t^\frac23 +  U), \ao ).
\end{align}
The spectrum of the Airy operator is exactly the zero set of the Airy function on $ \R $ up to a spatial reversal,
and the real zeros of Airy function admit precise asymptotic expansions (see, e.g., \cite[Section~11.5]{olver97}).
In particular, $ N( \lambda, \ao ) \leq c\, (\lambda_+)^{3/2} $, for all $ \lambda \in\R $.
Combining this with~\eqref{eq:Nsao:law} and	\eqref{eq:Nsao:law:},
we have that
\begin{align}
	\label{eq:Nsao:bd:}
	\exp\Big( L \,\int_\R \calM_{i_*+1}(\lambda) \cost'_t(\lambda) \d \lambda \Big)
	\geq
	\exp\Big( c \int_{\R} (\lambda-t^\frac23+U_*)^{\frac32}_+ \cost'_t(\lambda) \d \lambda \Big),
\end{align}
for some $ U_* \stackrel{\text{law}}{=} U $.
Consider the event $ \Omega_4 := \{ U_* \leq t^{\frac23} \} $.
Indeed, since $ U_* \stackrel{\text{law}}{=} U $ is $ [0,\infty) $-valued,
we have $ \Pr[\Omega_4] \to 1, $ as $ t\to\infty $.
On the r.h.s.\ of~\eqref{eq:Nsao:bd:}, using $ \cost'_t(\lambda) \geq - t^\frac13 e^{-t^\frac13\lambda} $ (verified from~\eqref{eq:costt})
and perform the change of variable $ \lambda-t^\frac23+U_* \mapsto \lambda $. 
Under the condition $ \Omega_4 := \{ U_* \leq t^{\frac23} \} $, we have
\begin{align}
	\notag
	\ind_{\Omega_4}
	\cdot
	\exp\Big( L \,\int_\R \calM_{i_*+1}(\lambda) \cost'_t(\lambda) \d \lambda \Big)
	&\geq
	\ind_{\Omega_4}
	\cdot
	\exp\Big( -c \int_{0}^\infty \lambda^\frac32 t^\frac13 e^{-t^\frac13(\lambda+t^\frac23-U_*)}) \d \lambda \Big)	
\\	
	\label{eq:Nsao:bd}
	&\geq
	\exp\Big( - c \int_{0}^\infty \lambda^\frac32 t^\frac13 e^{-t^\frac13\lambda} \d \lambda \Big)
	\geq 
	\frac12,
\end{align}
for all $ t $ large enough.

Next we turn to the exponential martingale in~\eqref{eq:lwbd:3}.
Recall that $ U_i := \max_{x\in\intv_i} |B(x)-B(\pt_{i-1})| $,
and that $ V(x) $ takes constant value $ V_i $ on $ \intv_i $,
and note from~\eqref{eq:V} that $ |V_i| \leq c t^{\frac23} $.
From thees properties we have
\begin{align*}
	\Big|\int_0^\infty V(x) \d B(x) \Big|
	\leq
	\sum_{i=1}^{i_*} |V_i| |B(\pt_{i})-B(\pt_{i-1})|
	\leq
	c t^\frac23 \sum_{i=1}^{i_*} U_i.
\end{align*} 
Using the condition~\eqref{eq:Omega3:b} together with $ i_* \leq c t^{\frac23-\alpha} $ gives
\begin{align}
	\label{eq:expmgbd}
	\ind_{\Omega_3(\kappa)} e^{ -\int_0^\infty V(x) \d B(x) }
	\geq
	\exp( - c t^{\frac43+\frac12(\delta+\alpha_+)} ).
\end{align} 

On the r.h.s.\ of~\eqref{eq:lwbd:3} withing the expectation, multiply by $ \ind_{\Omega_3(\kappa)\cap \Omega_4} $ to get
\begin{align*}
	G
	\geq
	e^{-ct^{\frac53-\alpha}-\frac12 \int_0^\infty V^2(x) \d x }
	\cdot
	\Ex \Big[
		\ind_{\Omega_2\cap\Omega_{3}(\kappa)\cap\Omega_4} 
		e^{-\int_0^\infty V(x) \d B(x)}
		\cdot 
		\prod_{i=1}^{i_*+1} \exp\Big( L \, \int_{\R} \calM_i(\lambda) \cost_{t}'(\lambda) \, \d\lambda \Big)
	\Big].
\end{align*}
On the r.h.s., insert the bounds \eqref{eq:calMibd}, \eqref{eq:Nsao:bd}--\eqref{eq:expmgbd}
(noting that $ M_i(\lambda,\kappa) $ is deterministic),
take logarithm, and divide the result by $ t^{2} $.
We obtain
\begin{align}
\label{eq:lwbd:5}
\begin{split}
	t^{-2} \log G
	\geq
	&-ct^{-\frac13-\alpha} -\frac12 \int_0^\infty t^{-2} V^2(x) \, \d x 
	- c t^{-\frac23+\frac12(\delta+\alpha_+)}
	+ L \, \sum_{i=1}^{i_*} \int_{\R} t^{-2} M_i(\lambda,\kappa) \cost_t'(\lambda) \, \d\lambda
\\
	&- t^{-2} \log 2
	+
	t^{-2}\log \Pr[\Omega_2\cap\Omega_3(\kappa)\cap\Omega_4].
\end{split}
\end{align}
As has been argued previously, $ \Pr[\Omega_3(\kappa)], \Pr[\Omega_4] \to 1 $, for fixed $ \kappa\in(0,\infty) $ as $ t\to\infty $.
As for $ \Omega_2 $, with $ V(x) \geq 0 $, comparison argument similarly to the preceding gives
$ \eigen_{1}(\sao+V) \geq \eigen_{1}(\sao) $.
This being the case, we necessarily have
$
	\Pr[\Omega_2] = \Pr[\eigen_{1}(\sao+V) > t^{-\frac23}]
	\geq
	\Pr[\eigen_{1}(\sao) > t^{-\frac23}] 
	\to
	1,
$
as $ t\to\infty $.
Consequently, $ \Pr[\Omega_2\cap\Omega_3(\kappa)\cap\Omega_4] \to 1 $.
Now, for fixed $ \kappa\in(0,\infty) $, sending $ t\to\infty $ in~\eqref{eq:lwbd:5},
together with $ \alpha>-\frac13 $ and $ \delta+\alpha_+ < \frac23 $, we arrive at
\begin{align}
\label{eq:lwbd::}
	\liminf_{t\to\infty} (t^{-2} \log G)
	\geq
	\liminf_{t\to\infty}
	\Big( -\frac12 \int_0^\infty t^{-2} V^2(x) \, \d x \Big)
	+
	\liminf_{t\to\infty}
	\Big( L \, \sum_{i=1}^{i_*} \int_{\R} t^{-2} M_i(\lambda,\kappa) \cost_t'(\lambda) \, \d\lambda \Big).
\end{align}

\medskip
\noindent\textbf{Step 4: evaluating the limit.}
The last step is to evaluate the limits on the r.h.s.\ of~\eqref{eq:lwbd::}.
For the first term, recall the definition of $ v_*(x) $ and $ V(x) $ from~\eqref{eq:v}--\eqref{eq:V}.
Substituting in $ |\intv_i| = t^\alpha $, we have
\begin{align*}
	\frac12 \int_0^\infty t^{-2} V^2(x) \, \d x 
	=
	\frac{t^{-2}}2 \sum_{i=1}^{i_*} t^\frac43 v^2_*(\pt_{i-1}t^{-\frac23}) \, |\intv_i|
	=
	\frac{1}2 \sum_{i=1}^{i_*} v^2_*((i-1)t^{\alpha-\frac23}) \, t^{\alpha-\frac23}. 
\end{align*}
The last expression is indeed a Riemann sum of the integral $ \frac12\int_0^\infty v^2_*(x) \d x $.
Since $ v_* $ is continuous and compactly supported, we have
\begin{align}
	\label{eq:v2int}
	\lim_{t\to\infty}
	\frac12 \int_0^\infty t^{-2} V^2(x) \, \d x 
	=
	\int_0^\infty \frac12 v^2_*(x) \d x.
\end{align}
Next, recall the definition of $ M_i(\lambda,\kappa) $ and $ r_i $ from~\eqref{eq:calMi}.
Indeed, the spectrum of the Laplace operator $ -\Delta_{\intv_i} $ is simply
$ \{ \eigen_{n}(-\Delta_{\intv_i}) \}_{n=1}^\infty = \{ n^2\pi^2 |\intv_i|^{-2} \}_{n=1}^\infty $.
Substituting in $ |\intv_i|=t^\alpha $, we obtain
\begin{align}
	\label{eq:Mi}
	M_i(\lambda,\kappa) 
	\leq 
	\frac{t^\alpha}{\pi}
	\sqrt{ \Big( \big(\tfrac{1+\kappa}{\kappa} \big)^3 \big( \lambda + r_i \big) + t^{\delta+\alpha_+} \Big)_+ }.
\end{align}
Apply $ \sum_{i=1}^{i_*} \int_{\R} t^{-2}(\, \Cdot \, )\cost_t'(\lambda) \, \d\lambda $ to both sides of~\eqref{eq:Mi}.
With $ \cost_t'<0 $, the resulting equality flip sides, giving
\begin{align*}
	L \, \sum_{i=1}^{i_*} \int_{\R} t^{-2} M_i(\lambda,\kappa) \cost_t'(\lambda) \, \d\lambda
	\geq
	\int_\R
	\frac{t^{\alpha-2}L}{\pi}
	\sum_{i=1}^{i_*}
	\sqrt{ \Big( \big(\tfrac{1+\kappa}{\kappa} \big)^3 \big( \lambda + r_i \big) + t^{\delta+\alpha_+} \Big)_+ }
	\,
	\cost'_t(\lambda) \, \d\lambda.
\end{align*}
Substitute in
$ r_i = t^\frac23\zeta-\tfrac{2}{\sqrt{\beta}}V_i-\pt_{i-1} $,
$ V_i = t^{\frac23} v_*((i-1)t^{\alpha-\frac23}) $, 
$ \pt_{i-1}=(i-1)t^\alpha $,
$ \cost'_t(\lambda) = -t^\frac13 e^{-t^\frac23\lambda}/(1+e^{-t^\frac23\lambda}) $,
and perform a change of variables $ t^{-\frac23} \lambda \mapsto \lambda $. 
We then obtain
\begin{align*}
	L \, \sum_{i=1}^{i_*} \int_{\R} t^{-2} &M_i(\lambda,\kappa) \cost_t'(\lambda) \, \d\lambda
	\geq
	-
	\frac{L}{\pi} 
	\int_\R
	\frac{ e^{-t\lambda}}{1+e^{-t\lambda}}
\\
	&\sum_{i=1}^{i_*} \sqrt{ \Big(
		\Big(\frac{1+\kappa}{\kappa}\Big)^3 \Big(\lambda + \zeta -\tfrac{2}{\sqrt{\beta}} v_*((i-1)t^{\alpha-\frac23}) -(i-1)t^{\alpha-\frac23} \Big) 
		+ t^{-\frac23+\delta+\alpha_+}
	\Big)_+} 
	\ t^{\alpha-\frac23} 
	\, \d\lambda.
\end{align*}
Given that $ \delta+\alpha_+ < \frac23 $, the term $ t^{-\frac23+\delta+\alpha_+} $ is vanishing as $ t\to\infty $.
Ignoring this term, we recognize the sum over $ i $ as a Riemann sum of $ (\frac{1+\kappa}{\kappa})^\frac32 \int_0^\infty \sqrt{(\lambda-\zeta+v_*(x))_+} d \lambda $.
On the other hand, as $ t\to\infty $, the factor $ 	\frac{ e^{-t\lambda}}{1+e^{-t\lambda}} \to \ind_{(-\infty,0)}(\lambda) $ for all $ \lambda\neq 0 $.
Hence, upon taking the limit $ t\to\infty $, we have
\begin{align}
	\notag
	\liminf_{t\to\infty} \sum_{i=1}^{i_*} \int_{\R} t^{-2} M_i(\lambda,\kappa) \cost_t'(\lambda) \, \d\lambda
	&\geq
	-
	\frac{L}{\pi} \Big(\frac{1+\kappa}{\kappa}\Big)^\frac32
	\int_{-\infty}^0 \int_0^\infty
	\sqrt{ (\lambda + \zeta - \tfrac{2}{\sqrt{\beta}} v_*(x) - x)_+ } \, \d\lambda \d x
\\	
	\label{eq:Mi:1}
	&=-
	\Big(\frac{1+\kappa}{\kappa}\Big)^\frac32
	\int_{0}^\infty
	\frac{2L}{ 3 \pi }  \big( \big(\zeta - \tfrac{2}{\sqrt{\beta}} v_*(x) - x \big)_+ \big)^\frac32 \, \d x.
\end{align}
Insert~\eqref{eq:v2int} and \eqref{eq:Mi:1} into~\eqref{eq:lwbd::}, and send $ \kappa\to\infty $.
We thus obtain
\begin{align}
\label{eq:lwbd:}
	\liminf_{t\to\infty} ( t^{-2} \log G )
	\geq
	-\int_0^\infty  \big( \tfrac12  v^2_*(x) + \tfrac{2L}{ 3 \pi }  \big( (\zeta - \tfrac{2}{\sqrt{\beta}}v_*(x) - x)_+ \big)^\frac32 \big) \, \d x.
\end{align}
It is readily checked from~\eqref{eq:v} that $ (\zeta - \tfrac{2}{\sqrt{\beta}}v_*(x) - x)_+ = (\frac{\sqrt{\beta}{\pi}}{2L} v_*(x))^2 $.
Using this this to substitute the $ \frac32 $-power in~\eqref{eq:lwbd:},
after straightforward but tedious calculations, we arrive at the desired lower bound:
\begin{align}
	\label{eq:lwbd}
	\liminf_{t\to\infty} ( t^{-2} \log G )
	\geq
	-\int_0^\infty  \Big( \frac12  v^2_*(x) + \frac{2}{ 3 \pi L }  \Big( \frac{\sqrt{\beta}\pi}{2L} v_*(x) \Big)^3 \Big) \, \d x
	=
	- L\Big(\frac{2L}{\beta}\Big)^5 \rate\Big(-\Big(\frac{\beta}{2L}\Big)^2\zeta\Big).
\end{align}

\subsection{Upper bound}
\label{sect:pfup}
First, from~\eqref{eq:costt}, it is readily checked that $ \cost_t(\lambda) \geq t^\frac13 \lambda_- $.
Using this, in~\eqref{eq:goal} we replace $ \cost_{t}(\eigen_{k}(\sao)-t^\frac23\zeta) $
with $ t^\frac13(t^\frac23\zeta-\eigen_{k}(\sao))_+ $ to get
\begin{align*}
	G 
	\leq 
	\Ex \Big[ \exp\Big( -L\,\sum_{k=1}^\infty t^\frac13(t^\frac23\zeta-\eigen_{k}(\sao))_+ \Big) \Big]
	=
	\Ex \Big[ \exp\Big( -L\,\int_{\R} t^\frac13(t^\frac23\zeta-\lambda)_+ \, \d \Nsao{\lambda} \Big) \Big].
\end{align*}
After performing integration by parts in $ \lambda $ and the decomposition~\eqref{eq:Nsao:decmp}, we have
\begin{align*}
	G
	\leq
	\Ex \Big[ \exp\Big( -t^\frac13L\sum_{i=1}^{i_*+1} \int_{-\infty}^{0} \Nsaoi{\lambda+t^\frac23\zeta}{i} \, \d\lambda \Big) \Big]
	\leq
	\Ex \Big[ \exp\Big( -t^\frac13L\sum_{i=1}^{i_*} \int_{-\infty}^{0} \Nsaoi{\lambda+t^\frac23\zeta}{i} \, \d\lambda \Big) \Big].
\end{align*}
Within the last expression, apply the bounds from Lemma~\ref{lem:loc}
to pass from $ \Nsaoi{\lambda+t^\frac23\zeta}{i} $ to $ \Nhill{\lambda-\pt_{i}+t^\frac23\zeta}{\intv_i} $.
Since the processes $ \Nhill{\Cdot}{\intv_i} $, $ i=1,\ldots,i_* $, are independent,
the resulting bound factorizes
\begin{align}
	\label{eq:goali}
	G
	\leq
	\prod_{i=1}^{i_*} G_i,
	\quad
	G_i:=
	\Ex \Big[ \exp\Big( -t^\frac13 L\int_{-\infty}^{0}\Nhill{\lambda-\pt_{i}+t^\frac23\zeta}{\intv_i} \, \d\lambda \Big) \Big].
\end{align}

Our next step is to bound each $ G_i $ in~\eqref{eq:goali}.
Fix hereafter $ i\in\{1,\ldots,i_*\} $, and, to simplify notation, we will often omit dependence on $ i $ in notation,
e.g., $ \intv=\intv_i $.
To begin with, using
\begin{align}
	\label{eq:trunc:sum}
	-t^\frac13L \int_{-\infty}^{0} N(\lambda+r,\hill_{\intv}) \, \d \lambda
	=
	-t^\frac13L \int_{\R} (r-\lambda)_+\, \d N(\lambda,\hill_{\intv})
	=
	-t^\frac13L \sum_{n=1}^\infty \big( r - \eigen_{n}(\hill_{\intv,\ell}) \big)_+,
\end{align}
we rewrite the term $ G_i $ as
\begin{align}
	\label{eq:goalii}
	G_i=
	\Ex \Big[ \exp\Big( -t^\frac13 L\,\sum_{n=1}^\infty \big( t^\frac23\zeta - \pt_{i} - \eigen_{n}(\hill_{\intv}) \big)_+ \Big].
\end{align}

Recall that $ \hill_\intv $ is constructed with Dirichlet boundary condition.
%
%
We will also need to consider operators with \emph{period} and \emph{Neumann} boundary conditions.
To setup notation for this, identify $ \intv=(\pt_{i-1},\pt_{i}] $ with the torus $ \T := \R/(|\intv|\Z) $,
and consider the Hilbert spaces $ H^1(\T) $ and $ H^1(\intv) $.
It is standard to check that  $ Q_{B} $ (defined in~\eqref{eq:schro:form} for $ \pot=B $) defines a coercive form,
both with respect to $ H^1(\T) \subset L^2(\intv) $ and with respect to $ H^1(\intv) \subset L^2(\intv) $.
Given this, we let 
$ \hill_{\T} $ and $ \hill_{\neu} $ be the associated operators of $ Q_B $
with respect to $ H^1(\T) \subset L^2(\intv) $ and $ H^1(\intv) \subset L^2(\intv) $, respectively:
\begin{align*}
	\hill_{\T} &:= -\frac{\d^2~}{\d x^2} + \frac{2}{\sqrt{\beta}} B'(x),
	\quad
	x\in \T,
\\
	\hill_{\neu} &:= -\frac{\d^2~}{\d x^2} + \frac{2}{\sqrt{\beta}} B'(x),
	\quad
	x\in \intv, \text{ with Neumann B.C.}
\end{align*}

\begin{remark}
\label{rmk:neu}
At first glance it may seem that the Hilbert space $ \hilbb=H^1(\intv) $ for $ \hill_{\neu} $ does not capture Neumann boundary condition,
but in fact any eigenfunction $ g $ of $ \hill_{\neu} $ does satisfy $ g'(\pt_{i-1})=g'(\pt_{i})=0 $.
To see this, consider an eigenvalue problem for $ \hill_{\neu} $: a given function $ g\in H^1(\intv) $ and $ \lambda\in\R $ satisfying
\begin{align}
	\label{eq:neueigenvalue}
	\int_{\intv} \Big( \frac{1}{2} g'(x) p'(x) + \frac{2}{\sqrt{\beta}} g(x) p(x) B'(x) - \lambda g(x) p(x) \Big) \d x =0,
	\quad
	\forall p \in H^1(\intv).
\end{align}
Given that $ B $ is $ a $-H\"{o}lder continuous for $ a<\frac12 $,
it is standard to show that $ g' $ is also $ a $-H\"{o}lder continuous for $ a<\frac12 $,
so in particular $ g'(\pt_{i-1}) $ and $ g'(\pt_{i}) $ are well-defined.
Now, for the test function $ p(x) = p_\delta(x) := (1-\delta^{-1}(x-\pt_{i-1}))_+ $, 
using $ g,g'\in C(\intv) $, it is readily checked that
\begin{align*}
	\lim_{\delta\to 0} \int_{\intv} g'(x) p_\delta'(x) \d x &= -g'(\pt_{i-1}),
\\
	\lim_{\delta\to 0} \int_{\intv} g(x) p_\delta(x) \d x &= 0,
\\
	\lim_{\delta\to 0} \int_{\intv}  g(x) p_\delta(x) B'(x) \d x
	&:=
	\lim_{\delta\to 0} \Big(  
		g(x)p_\delta(x) B(x) \big|_{\pt_{i-1}}^{\pt_{i}}
		-\int_{\intv}  \big( g'(x) p_\delta(x) + g(x) p'_\delta(x) \big) B(x) \d x
	\Big)
\\
	&= -g(\pt_{i-1}) B(\pt_{i-1}) + g(\pt_{i-1}) B(\pt_{i-1}) =0.
\end{align*}
Combining these properties with~\eqref{eq:neueigenvalue} yields $ f'(\pt_{i-1})=0 $.
A similar procedure applied to the test function $ (1-\delta^{-1}(\pt_{i}-x))_+ $ yields $ g'(\pt_{i})=0 $.
\end{remark}

To bound the r.h.s.\ of~\eqref{eq:goalii}, our first step is to pass from $ \hill_\intv $ to $ \hill_\T $ and $ \hill_\neu $.
\begin{lemma}
\label{lem:periodic}
Almost surely for all $ r\in\R $,
\begin{align}
	\label{eq:periodic}
	\sum_{n=1}^\infty \big( r - \eigen_{n}(\hill_{\intv}) \big)_+
	\leq
	\eigen_{1}(\hill_\neu)-\sum_{n=1}^\infty \big( r - \eigen_{n}(\hill_{\T}) \big)_+.
\end{align}
\end{lemma}
\begin{proof}
Fix a mollifier $ q $, namely $ q\in C^\infty(\R) $, supported in $ (-1,1) $, $ q \geq 0 $, and $ \int_{\R}q(x)\d x =1 $.
For $ \e>0 $, mollify the Brownian motion $ B_\e(x) := \int_{\R} q(\e^{-1}y)  B(x-y) \e^{-1} \d y \in C^\infty(\intv) $.
Accordingly, let $ \hill_{\intv,\e} $ and $ \hill_{\T,\e} $ be the associated operators of $ Q_{B_\e} $
with respect to $ H^1(\T) \subset L^2(\intv) $ and $ H^1(\intv) \subset L^2(\intv) $, respectively.
A classical result \cite[Equation~(3.15), Proof of Theorem~8.3.1]{coddington55} of Sturm--Liouville theory asserts that, 
for operators the form~\eqref{eq:schro} with piecewise continuous $ \pot'(x) $,
the eigenvalues under Dirichlet and under periodic boundary conditions interlace.
Applying this result with $ \pot = B_\e $ gives
\begin{align}
	\label{eq:interlace:}
	-\infty<
	\eigen_{1}(\hill_{\T,\e}) \leq \eigen_{1}(\hill_{\intv,\e})
	\leq
	\eigen_{2}(\hill_{\T,\e}) \leq \eigen_{2}(\hill_{\intv,\e})
	\leq
	\eigen_{3}(\hill_{\T,\e}) \leq \eigen_{3}(\hill_{\intv,\e})
	\leq
	\ldots \to\infty.
\end{align}
Our next step is to pass~\eqref{eq:interlace:} to the limit $ \e\to 0 $.
Indeed, almost surely for all $ \e\in(0,1) $,
we have $ \sup_{x\in\intv} |B_\e(x)| \leq \sup_{x\in[\pt_{i-1}-1,\pt_i+1]} |B(x)| < \infty $.
Also, as $ \e\to 0 $, we have $ \sup_{x\in\intv} |B_\e(x)-B(x)| \to_\text{P} 0 $.
Given these properties, 
apply the bounds from Lemma~\ref{lem:spec:cmp} with $ (\pot_1,\pot_2)=(B,B_\e) $ and with $ (\pot_1,\pot_2)=(B_\e,B) $.
Sending $ \e\to 0 $ and $ \kappa\to\infty $ in order, we obtain that
$ \eigen_{n}(\hill_{\intv,\e}) \to_\text{P} \eigen_{n}(\hill_{\intv}) $, for any $ n\in\N $ as $ \e\to\infty $.
Similar argument applied to periodic boundary condition gives $ \eigen_{n}(\hill_{\T,\e}) \to_\text{P} \eigen_{n}(\hill_{\T}) $.
Now taking the limit $ \e\to\infty $ in~\eqref{eq:interlace:} gives
\begin{align}
	\label{eq:interlace}
	-\infty<
	\eigen_{1}(\hill_{\T}) \leq \eigen_{1}(\hill_{\intv})
	\leq
	\eigen_{2}(\hill_{\T}) \leq \eigen_{2}(\hill_{\intv})
	\leq
	\eigen_{3}(\hill_{\T}) \leq \eigen_{3}(\hill_{\intv})
	\leq
	\ldots \to\infty.
\end{align}

The interlacing condition~\eqref{eq:interlace} gives, for any $ r\in\R $,
\begin{align}
	\label{eq:BCcmp0}
	-\sum_{n=1}^\infty \big(r- \eigen_{n}(\hill_\intv) \big)_+
	\leq
	-\sum_{n=2}^\infty \big(r- \eigen_{n}(\hill_\T) \big)_+
	=
	\big(r- \eigen_{1}(\hill_\T) \big)_+
	-\sum_{n=1}^\infty \big(r- \eigen_{n}(\hill_\T) \big)_+.
\end{align}
On the other hand, since $  H^1(\T) \subset H^1(\intv) $,
applying the minimax principle~\eqref{eq:minimax} for $ k=1 $ and for $ T=\hill_\T, \hill_\intv $, we have
$
	\eigen_{1}(\hill_{\neu}) \leq \eigen_{1}(\hill_{\T}).
$
Using this in~\eqref{eq:BCcmp0} to bound 
$ (r- \eigen_{1}(\hill_\T) )_+ \leq (r- \eigen_{1}(\hill_\neu) )_+ $,
we conclude the desired result.
\end{proof}

We now direct our attention to the last sum in~\eqref{lem:periodic}.
The next proposition is the key step of the proof.
\begin{proposition}
\label{prop:key}
Set $ \lambda^*_n  := ( 2\pi |\intv|^{-1} \lfloor\frac{n}{2}\rfloor)^2. $
Almost surely for all $ r\in\R $, 
\begin{align}
	\label{eq:key}
	-\sum_{n=1}^\infty \big( r - \eigen_{n}(\hill_{\T}) \big)_+
	\leq
	-\sum_{n=1}^\infty \Big( r- \frac{2}{\sqrt{\beta}}\frac{B(\pt_{i})-B(\pt_{i-1})}{|\intv|} +\lambda_n^* \Big)_+.	
\end{align}
\end{proposition}
\begin{proof}
The readily checked identity that `removes the $ + $' will be useful:
\begin{align}
	\label{eq:simple}
	-\sum_{n=1}^\infty \big( x_n \big)_+
	=
	-\sup_{m\in\Z_{\geq 0}} \Big\{ \sum_{n=1}^m x_n \Big\}
	=
	\inf_{m\in\Z_{\geq 0}} \Big\{ -\sum_{n=1}^m x_n \Big\},
	\quad
	\text{ for any }
	\infty > x_1 \geq x_2 \geq x_3 \geq \ldots,
\end{align}
with the convention that empty sum is zero.
Now, consider the Fourier basis of $ L^2(\T) $:
\begin{align*}
	f_1(x) := |\intv|^{-\frac12}, 
	\quad
	f_{2k}(x) :=  |\intv|^{-\frac12} e^{\img \frac{2\pi k}{|I|} }, 
	\ 
	f_{2k+1}(x) :=  |\intv|^{-\frac12} e^{-\img \frac{2\pi k}{|I|} },
	\quad
	k=1,2,\ldots.
\end{align*}
Set $ b := \frac{1}{|\intv|}(B(\pt_{i})-B(\pt_{i-1})) $ to simplify notation.
Insert these vectors $ f_n $ into the form $ Q_{B} $ (defined in~\eqref{eq:schro:form} for $ \pot=B $) and sum the result over $ n=1,\ldots,m $.
With $ |f_n(x)|^2 \equiv \frac{1}{|\intv|} $ and with $ \int_{\pt_{i-1}}^{\pt_i} B'(x) \d x = b  $, we have
\begin{align*}
	\sum_{n=1}^m  Q_{\hill_{\T,\ell}}(f_n,f_n)
	=
	\sum_{n=1}^m \Big( \int_\T |f'_n(x)|^2 \d x + \frac{2}{\sqrt{\beta}} \int_{\pt_{i-1}}^{\pt_{i}} |f_n(x)|^2 B'(x) \d x \Big)
	=
	\sum_{n=1}^m \big( \lambda_n^* + \tfrac{2}{\sqrt{\beta}} b \big).	
\end{align*}
Since $ \{f_1,\ldots,f_m\} \subset H^1(\T) $ is orthonormal in $ L^2(\T) $, Lemma~\ref{lem:minimaxx} gives
$ \sum_{n=1}^m \eigen_{n}(\hill_\T) \leq \sum_{n=1}^m ( \lambda_n^* + \frac{2}{\sqrt{\beta}}b) $, or equivalently
\begin{align*}
	-\sum_{n=1}^m \big( r - \eigen_{n}(\hill_{\T}) \big)
	\leq
	-\sum_{n=1}^m \Big( r- \tfrac{2}{\sqrt{\beta}} b -\lambda_n^* \Big).
\end{align*}
Applying~\eqref{eq:simple} with $ x_n = r -\eigen_{n}(\hill_\T) $, we have
\begin{align*}
	-\sum_{n=1}^\infty \big( r - \eigen_{n}(\hill_{\T}) \big)_+
	\leq
	-\sum_{n=1}^m \big( r - \eigen_{n}(\hill_{\T}) \big)
	\leq
	-\sum_{n=1}^m \Big( r- \tfrac{2}{\sqrt{\beta}} b -\lambda_n^* \Big),
\end{align*}
for any $ m\in\Z_{\geq 0} $.
Since this holds for all $ m\in\Z_{\geq 0} $, optimizing over $ m $, 
and then applying~\eqref{eq:simple} with $ x_n = r-\frac{2}{\sqrt{\beta}} b-\lambda_n^* $ in reverse,
we conclude the desired result.
\end{proof}

Write $ |\intv|^{-1}(B(\pt_{i})-B(\pt_{i-1})) := t^{-\frac{\alpha}{2}} Z $, so that $ Z $ is a standard Gaussian.
Recall the gives expression~\eqref{eq:goalii} of $ G_i $.
Combine Lemma~\ref{lem:periodic} with Proposition~\ref{prop:key} for $ r=t^\frac23\zeta-\pt_{i} $.
Multiply the result by $ t^\frac13 L $, exponentiate, and take $ \Ex[\,\Cdot\,] $.
With $ (r-\eigen_{1}(\hill_{\neu}))_+ \leq r_+ + (\eigen_{1}(\hill_{\neu}))_- $,
we have
\begin{align*}
	G_i
	\leq
	\Ex \Big[ \exp \Big( ct + c t^\frac13 (\eigen_{1}(\hill_{\neu}))_- -t^\frac13L\,\sum_{n=1}^\infty \big( t^\frac23\zeta-\pt_{i} - t^{-\frac{\alpha}{2}} Z -\lambda_n^* \big)_+ \Big) \Big].	
\end{align*}
Fix an auxiliary parameter $ \kappa\in[1,\infty) $.
To separate terms within the last expression,
we apply H\"{o}lder's inequality with exponents $ \kappa+1 $ and $ \frac{\kappa+1}{\kappa} $ to get
\begin{align}
	\label{eq:Gibd}
	G_i
	\leq
	e^{ct} G_{i,1}^{\frac{1}{\kappa+1}} G_{i,2}^\frac{\kappa}{\kappa+1},
\end{align}
where
\begin{align*}
	G_{i,1} &:= \Ex\Big[ \exp \Big( c\,t^\frac13 (\kappa+1)(\eigen_{1}(\hill_{\neu}))_- \Big) \Big],
\\
	G_{i,2} &:=
	\Ex \Big[
		\exp\Big( -t^\frac13 L \frac{\kappa+1}{\kappa} \sum_{n=1}^\infty \big( t^\frac23\zeta-\pt_{i} - t^{-\frac{\alpha}{2}} Z -\lambda_n^* \big)_+ \Big) 
	\Big].
\end{align*}
We now proceed to bound the terms $ G_{i,1} $ and $ G_{i,2} $.

\begin{lemma}
\label{lem:G1bd}
For all $ t\geq 1 $, we have
$
	\log (G_{i,1})
	\leq
	c \, (\kappa+1)^3 t.
$
\end{lemma}

\noindent
The proof of Lemma~\ref{lem:G1bd} goes through a series of comparison argument for Riccati-type ODE's.
As the argument is rather disjoint from the rest of the proof,
to avoid breaking the flow, we postpone proving Lemma~\ref{lem:G1bd} till the end of this subsection.
As for the term $ G_{i,2} $, recall the definition of $ v_* $ from~\eqref{eq:v}.
$ \til{v}_i(\kappa) := \til{v}(\pt_it^{-\frac23},\kappa) $.
\begin{lemma}
\label{lem:G2bd}
For all $ \kappa>0 $ and $ t<\infty $,
\begin{align*}
	&
	\log G_{i,2}
	\leq
	- 
	t^{\alpha+\frac43}
	\Big( 
		\frac{2L}{3\pi} \big( \zeta-t^{-\frac23}\pt_{i}+ v_*(t^{-\frac23}\pt_{i}) \big)^{\frac{3}{2}} + \frac{1}{2} v^2_*(t^{-\frac23}\pt_{i})
	\Big)
	+
	c\,(\kappa+1)^2 t^{\frac13-2\alpha}.
\end{align*}
\end{lemma}
\begin{proof}
Recall that $ \lambda^*_n  := ( 2\pi |\intv|^{-1} \lfloor\frac{n}{2}\rfloor)^2 $.
Forgoing the first eigenvalue $ \lambda^*_1 $, we write
\begin{align*}
	-\sum_{n=1}^\infty (r-\lambda^*_n)_+
	\leq
	-\sum_{n=2}^\infty (r-\lambda^*_n)_+
	=
	-\sum_{k=1}^\infty 2(r-4\pi^2|\intv|^{-2} k^2)_+.
\end{align*}
Since $ (r-4\pi^2|\intv|^{-2} x^2)_+ $ is a decreasing function of $ x $ for $ x \geq 0 $, comparing sums to integrals gives,
for $ y_0 := 4\pi|\intv|^{-1} $,
\begin{align*}
	-\sum_{n=1}^\infty (r-\lambda^*_n)_+
	\leq
	-2 \int_2^\infty (r-4\pi^2|\intv|^{-2} x^2)_+ \d x
	=&
	\frac{|\intv|}{\pi} \Big( -\frac23 r^\frac32 + ry_0 - \frac13 y_0^3 \Big) \ind_\set{r>y_0^2}.
\end{align*}
Within the last expression, drop the $ - \frac13 y_0^3 $ term,
and divide $ -\frac23 r^\frac32 $ into `two pieces' to get
\begin{align*}
	-\sum_{n=1}^\infty (r-\lambda^*_n)_+
	\leq
	\frac{|\intv|}{\pi} \Big( -\frac{2\kappa}{3(1+\kappa)} r^\frac32 -\frac{2}{3(1+\kappa)} r^\frac32 + ry_0  \Big) \ind_\set{r>y_0^2}.
\end{align*}
Consider separately the cases
$
	\frac{2}{3(1+\kappa)} r^\frac32 \geq  ry_0 
$
and 
$
	\frac{2}{3(1+\kappa)} r^\frac32 <  ry_0,
$
we then have
\begin{align}
	\label{eq:Gi2:1.2}
	-\sum_{n=1}^\infty (r-\lambda^*_n)_+
	\leq
	\frac{|\intv|}{\pi} \Big( -\frac{2\kappa}{3(1+\kappa)} r_+^\frac32 - c\,(1+\kappa)^2 y^3_0  \Big)
	\leq
	-\frac{2\kappa|\intv|}{3(1+\kappa)\pi} r_+^\frac32 + c\, (1+\kappa)^2 |\intv|^{-2}.  
\end{align}
Within~\eqref{eq:Gi2:1.2}, substitute $ r=t^\frac23\zeta-\pt_{i} - t^{-\frac{\alpha}{2}} Z $ and $ |\intv|=t^\alpha $,
multiply the result by $ t^\frac13 \frac{\kappa+1}{\kappa} $, exponentiate, and take $ \Ex[\,\Cdot\,] $.
We have
\begin{align}
	\label{eq:Gi2:2}
	G_{i,2} \leq
	e^{c\,(\kappa+1)^2t^{\frac13-2\alpha}}
	\Ex\Big[ \exp\Big( 
		- \frac{2L}{3\pi} t^{\frac13+\alpha} \big( t^\frac23\zeta-\pt_{i} - t^{-\frac{\alpha}{2}} Z  \big)^{\frac32}_+ 
	\Big)\Big].
\end{align}

Recall that $ Z $ is a standard Gaussian. We then evaluate the expectation on the r.h.s.\ of~\eqref{eq:Gi2:2} as 
\begin{align*}
	\int_\R \frac{e^{-F(y)}}{\sqrt{2\pi}} \d y,
	\quad
	F(y) 
	:=  
	\frac{2L}{3\pi} t^{\frac13+\alpha} \big( t^\frac23\zeta-\pt_{i} - t^{-\frac{\alpha}{2}} y  \big)^{\frac32}_+ + \frac12 y^2.
\end{align*}
Indeed, $ F $ is $ C^\infty $ except at the point $y_\text{c} $ where $ t^\frac23\zeta-\pt_{i} - t^{-\frac{\alpha}{2}}y_\text{c}=0 $,
and at $y_\text{c} $, $ F $ is still $ C^1 $.
Given these properties, straightforward differentiations show that $ F(y) $ reaches its global minimum at $ y_*:=t^{\frac23+\frac{\alpha}{2}} v_*(t^{-\frac23}\pt_{i}) $,
and $ F''(y) \geq 1 $ expect at $ y=y_\text{c} $.
Consequently, $ F(y) \geq F(y_*) + \frac12 (y-y_*)^2 $, which gives
\begin{align*}
	\int_\R \frac{e^{-F(v)}}{\sqrt{2\pi}} \d v
	\leq
	\exp(-F(v_*)) 
	= 
	\exp\Big(
	-t^{\alpha+\frac43}
	\Big( 
			\frac{2L}{3\pi} \big( \zeta-t^{-\frac23}\pt_{i} +v_*(t^{-\frac23}\pt_{i}) \big)^{\frac{3}{2}}_+ + \frac{1}{2} v^2_*(t^{-\frac23}\pt_{i})
	\Big)\Big). 
\end{align*}
Combining this with~\eqref{eq:Gi2:2} gives the desired result.
\end{proof}

Now, rewrite~\eqref{eq:goali}--\eqref{eq:Gibd} as
$
	\log G
	\leq
	\sum_{i=1}^{i_*} \log G_i
	\leq
	ct + \sum_{i=1}^{i_*} (\frac{1}{\kappa+1}\log G_{i,1} + \frac{\kappa}{\kappa+1} \log G_{i,2}).
$
Then, insert the bounds from Lemmas~\ref{lem:G1bd}--\ref{lem:G2bd}, and divide the result by $ t^2 $.
With $ i_* \leq c t^{\frac23-\alpha} $, we arrive at
\begin{subequations}
\begin{align}
	\label{eq:upbd:sum:a}
	t^{-2} \log G
	\leq
	&c \, \Big( 
		t^{-1}
		+
		(\kappa+1)^2 t^{-\frac13-\alpha}
		+
		(\kappa+1)^2 t^{-1-3\alpha}
	\Big)
\\
	\label{eq:upbd:sum:b}
	&-
	\sum_{i=1}^{i_*}
	\frac{\kappa}{\kappa+1}
	\Big( \frac{2L}{3\pi} \big( \zeta-t^{-\frac23}\pt_{i} +v_*(t^{-\frac23}\pt_{i}) \big)^{\frac{3}{2}}_+ + \frac{1}{2} v^2_*(t^{-\frac23}\pt_{i}) \Big) t^{-\frac23+\alpha}.
\end{align}
\end{subequations}
With $ \alpha\in(-\frac13,\frac23) $,
the the r.h.s.\ vanishes as $ t\to\infty $.
Recognizing the term in~\eqref{eq:upbd:sum:b} as a Riemann sum (as done in Section~\ref{sect:pflw}),
sending $ t\to\infty $ and $ \kappa\to\infty $ in order, we obtain 
\begin{align*}
	\limsup_{t\to\infty }
	\big( t^{-2} \log G \big)
	\leq
	-
	\int_{\R} \frac{2L}{3\pi} \Big( (\zeta-x+v_*(x))^{\frac{3}{2}}_+ + \frac{1}{2} v^2_*(x) \Big) \d x.
\end{align*}
The last expression matches the previously established lower bound~\eqref{eq:lwbd}. 
The proof is now completed upon settling Lemma~\ref{lem:G1bd}.

\begin{proof}[Proof of Lemma~\ref{lem:G1bd}]
Throughout the proof, we write $ \eigen_{1}=\eigen_{1}(\hill_{\neu}) $ to simplify notation.
Recall that $ i $ indexes which interval $ \intv=\intv_i $ we are considering.
The law of $ \eigen_{1} $ is clearly independent of $ i $, so, without lost of generality, we take $ i=1 $,
and $ \intv=\intv_1 = (0,\pt_1] $.

The proof amounts to establishing a suitable tail bound on $ (\eigen_{1})_- $.
We achieve this by a series of comparison of the Riccati equation~\eqref{eq:Riccati}.
Recall that our discussion regarding~\eqref{eq:Riccati} in Section~\ref{sect:tool} is \emph{pathwise},
and holds for every realization (i.e., any $ C[0,\pt_{1}] $ function) of $ B $.
On the other hand, within this proof we will also regard~\eqref{eq:Riccati} as a \ac{SDE}
\begin{align}
	\label{eq:Riccati:SDE}
	\d f(x) = (-\lambda - f^2(x)) \d x + \tfrac{2}{\sqrt{\beta}} \d B(x),
\end{align}
and, accordingly, sometimes view $ f $ as a process.
It is standard to check that $ f $ satisfies the strong Markov property.
That is, letting $ \mathscr{F}(x) := \sigma(B(y):y\geq 0) $ denote the canonical filtration of $ B $,
and $ f^a(x) $ denote the solution of~\eqref{eq:Riccati:SDE} with initial condition $ f(0)=a $,
then, for any $ \mathscr{F} $-stopping time $ \tau $, we have
\begin{align*}
	f(\Cdot+\tau) \stackrel{\text{law}}{=}
	f^{f(\tau)}(\Cdot).
\end{align*}
Let $ f(x,\lambda) $ denote the solution of~\eqref{eq:Riccati} with initial condition $ f(0,\lambda)=0 $,
and let $ \tau(\gamma;g) := \inf\{ x\in[0,\pt_1] : g(x)=\gamma \} $ 
denote the first hitting time of a given function $ g $ at level $ \gamma $,
with the convention that $ \inf\emptyset :=\infty $.
To simplify notation we write $ \tau_{\pm,s} := \tau(-\frac12\sqrt{s};f(\Cdot,-s)) $.

The proof is carried out in steps.

\medskip
\noindent\textbf{Step 1: truncation.}
This step of the proof follows similar arguments in~\cite{dumaz13}.
In this step we establish a useful truncation bound~\eqref{eq:G1bd:key} that allows use to restriction our attention
to the band $ f(x,-s)\in[-\frac12\sqrt{s},\frac12\sqrt{s}] $.
To setup notation, let
\begin{align*}
	\Omega_{-+} := \{\tau_{-,s}<\tau_{+,s} \},
	\quad
	\Omega_{+-} := \{ \tau_{+,s}<\tau_{-,s} \}.
\end{align*}
For $ s \geq t^{\alpha_+} $, we aim at showing
\begin{align}
	\label{eq:G1bd:key}
	\Pr\big[ \tau_{-,s} < \infty \big] 
	\leq 
	c\, \Pr\big[ \{\tau_{-,s} < \infty \} \cap \Omega_{-+} \big].
\end{align}
Decompose the l.h.s.\ of~\eqref{eq:G1bd:key} into
\begin{align}
	\label{eq:G1bd:decmp}
	\Pr\big[ \tau_{-,s} < \infty \big] 
	=
	\Pr\big[  \tau_{-,s} < \infty, \  \Omega_{-+} \big]
	+ 
	\Pr\big[  \tau_{-,s} < \infty, \ \Omega_{+-}  \big].
\end{align}
The last term in~\eqref{eq:G1bd:decmp} encodes the probability that $ f(x,-{s}) $, which starts at $ f(0,-{s})=0 $,
first hits level $ \frac12\sqrt{s} $, and then hits level $ -\frac12\sqrt{s} $.
Reinitiate the process $ f(x,-{s}) $ at $ x=\tau_{+,s} $,
the strong Markov property gives $ f(\Cdot+\tau_{+,s}) \stackrel{\text{law}}{=} f_1(\Cdot) $, where
$ f_1 $ solves~\eqref{eq:Riccati:SDE} for $ \lambda=-{s} $ with the initial condition $ f_1(0)=\frac12\sqrt{s} $.
This gives
\begin{align}
	\label{eq:G1bd:t1}
	\Pr\big[ \tau_{-,s} < \infty,\ \Omega_{+-} \big]
	\leq
	\Pr\big[ \tau(-\tfrac12\sqrt{s},f_1) < \infty \big].
\end{align}
The r.h.s.\ of~\eqref{eq:G1bd:t1} encodes the probability that $ f_1 $, 
which starts at $ f_1(0)=\frac12\sqrt{s} $, hits level $ -\frac12\sqrt{s} $ within $ x\in[0,\pt_{1}] $.
This being the case, $ f_1 $ must also have hit $ 0 $.
Reinitiate the process $ f_1(x) $ at $ x=\tau(0;f_1) $.
By the strong Markov property we have $ f_1(\Cdot+\tau(0;f_1)) \stackrel{\text{law}}{=} f(\Cdot,-{s}) $, so
\begin{align*}
	\Pr\big[ \tau_{-,s} < \infty,\ \Omega_{+-} \big]
	\leq
	\Pr\big[ \tau(0,f_1) < \infty \big]
	\cdot
	\Pr\big[ \tau_{-,s} < \infty \big].
\end{align*}
Combining this with~\eqref{eq:G1bd:decmp}--\eqref{eq:G1bd:t1} now gives
\begin{align}
	\label{eq:G1bd:decmp:}
	\Pr\big[ \tau_{-,s} < \infty \big] 
	=
	(1-R)^{-1} \Pr\big[ \tau_{-,s} < \infty,\ \Omega_{-+} \big],
\end{align}
where $ R:= \Pr[ \tau(0,f_1) < \infty ] $.

We proceed to bound $ R $.
To this end, consider the event $ D_0:= \{ \sup_{x\in[0,\pt_{1}]} \frac{2}{\sqrt{\beta}}|B(x)| \geq \frac14\sqrt{s} \} $.
Recall that $ f_1(0) = \frac12\sqrt{s} $.
Let $ \tau^*_1 := \sup\{ x\in[0,\tau(0,f_1)] : f_1(x) \geq \frac12\sqrt{s} \} $ 
be the last exist time of $ f_1 $ from the region above $ \frac12\sqrt{s} $ before $ f_1 $ hits level $ 0 $.
Under the occurrence of $ \{\tau(0,f_1)<\infty\} $,
setting $ (x_1,x_2) = (\tau^*_1, \tau(0,f_1)) $ in~\eqref{eq:Riccati:int} gives
\begin{align*}
	\text{On }\{\tau(0,f_1)<\infty\},
	\quad
	-\frac{\sqrt{s}}{2} = f(x)\Big|_{x=\tau^*_1}^{x=\tau(0,f_1)}
	=
	\int_{\tau^*_1}^{\tau(0,f_1)} (s - f^2(x)) \d x + \frac{2}{\sqrt{\beta}} B(x)\Big|_{x=\tau^*_1}^{x=\tau(0,f_1)}.
\end{align*}
On the r.h.s., the integral is nonnegative since $ (s-f_1^2(x)) \geq \frac34 s \geq 0 $ for $ x\in[\tau^*_1, \tau(0,f_1)] $.
This gives
\begin{align*}
	\big\{ \tau(0,f_1)<\infty \big\}
	\subset
	\Big\{ \tfrac{2}{\sqrt{\beta}} B(x)\big|_{x=\tau^*_1}^{x=\tau(0,f_1)} \leq - \tfrac12 \sqrt{s} \Big\}	
	\subset
	D_0
\end{align*}
and hence 
$
	R:= \Pr[ \tau(0,f_1) < \infty ]
	\leq
	\Pr[ D_0 ].
$
Under the assumption $ s \geq t^{\alpha_+} $, together with $ \pt_1 = t^\alpha $,
it is readily checked that $ \Pr[ D_0 ] \leq \frac{1}{c+1} $, for all $ t \geq 1 $.
Hence $ R \leq \frac{1}{c+1} $.
Inserting this bound into~\eqref{eq:G1bd:decmp:} gives~\eqref{eq:G1bd:key}.

\medskip
\noindent\textbf{Step~2: Reduction to Brownian exist probability.}
Fix $ s \geq t^{\alpha_+\vee(-2\alpha)} $.
Our goal in this step is to bound the tail probability.
To begin with, consider the associated eigenfunction $ g_* $ of $ \eigen_1 $.
Taking the real part of $ g_* $ if necessary, we may assume $ g_* $ is $ \R $-valued.
Referring to Remark~\ref{rmk:neu}, we have that $ g_* $ is in fact $ C^1 $ with $ g_*'(0)=g_*'(\pt_{1})=0 $.
Riccati transform $ f_*:= g_*'/g_* $ furnishes a solution of~\eqref{eq:Riccati} for $ \lambda=\eigen_{1} $ such that $ f_*(0)=f_*(\pt_1)=0 $.
On the event $ \{ \eigen_{1} \leq -s \} $ under current consideration,
Proposition~\ref{prop:Riccati}\ref{prop:Riccati:ode} asserts that 
$ f(x,-s) \leq f_*(x) $, $ \forall x\in\intv $, under the ordering described in Section~\ref{sect:tool}.
Consequently, either $ f(x,-s) $ hits the level $ -\frac12\sqrt{s} $ (which gives $ \tau_{-,s}<\infty $), 
or, if not, $ f(\pt_1,-s) \leq 0 $.
This gives
\begin{align*}
	\Pr\big[ \eigen_{1} < -s  \big]
	=
	\Pr\big[ \tau_{-,s} <\infty \big]
	+
	\Pr\big[ \tau_{-,s} = \infty, \ f(\pt_{1},-s) \leq 0 \big].	
\end{align*}
Apply~\eqref{eq:G1bd:key} to the first term on the r.h.s., we have
\begin{align}
	\label{eq:G1bd:step2}
	\Pr\big[ \eigen_{1} < -s  \big]
	\leq
	c\, \Pr\big[ \Omega_1 \big]
	+
	\Pr\big[ \Omega_2 \big],
\end{align}
where $ \Omega_1 := \{\tau_{-,s} <\infty\} \cap \Omega_{-+} $
and $ \Omega_2 := \{ \tau_{-,s} = \infty, \ f(\pt_{1},-s) \leq 0 \} $.

The next step is to bound the probability on the r.h.s.\ of~\eqref{eq:G1bd:step2}.
Under the occurrence of $ \Omega_1 $, 
set $ (x_1,x_2) = (0,\tau_{-,s}) $ and $ \lambda=-s $ in~\eqref{eq:Riccati:int} to get
\begin{align*}
	\text{On } \Omega_1 ,
	\quad
	-\frac{\sqrt{s}}2
	=
	f(\tau_{-,s})
	=
	\int_0^{\tau_{-,s}}(s-f^2(x)) \d x
	+
	\frac{2}{\sqrt{\beta}} B(\tau_{-,s}).
\end{align*}
Since $ |f(x)| \leq \frac12\sqrt{s} $ for all $ x \leq \tau_{+,s}\wedge \tau_{-,s} $,
here we have $ \int_0^{\tau_{-,s}}(s-f^2(x)) \d x \geq \frac34 s \tau_{-,s} $.
This gives
\begin{align}
	\label{eq:Omega-+}
	\text{On } \Omega_1 ,
	\quad
	-\tfrac{\sqrt{s}}2
	-
	\tfrac34  s \tau_{-,s}
	\geq
	\tfrac{2}{\sqrt{\beta}} B(\tau_{-,s}).
\end{align}
Consider further the sub-events
$
	\Omega_{1,\leq}
	:= 
	\{ \tau_{-,s} \leq s^{-\frac12} \}\cap \Omega_1
$
and
$
	\Omega_{1,>} 
	:= 
	\{ s^{-\frac12} <\tau_{-,s} <\infty  \}\cap \Omega_{1}.
$
Under the occurrence of $ \Omega_{1,\leq} $, forgoing the term $ -\frac34  s \tau_{-,s} $ in~\eqref{eq:Omega-+} gives
\begin{align}
	\label{eq:D<}
	\Omega_{1,\leq}
	\subset
	\Big\{
		\sup_{x\in[0,s^{-1/2}]} \frac{2}{\sqrt{\beta}} |B(x)|
		\geq
		\frac{\sqrt{s}}2
	\Big\}
	:=
	D_1(s).
\end{align}
Under the occurrence of $ \Omega_{1,>} $, forgoing the term $ -\tfrac{\sqrt{s}}2 $ in~\eqref{eq:Omega-+} gives
\begin{align}
	\label{eq:D>}
	\Omega_{1,\leq}
	\subset
	\Big\{
		\sup_{x\geq s^{-1/2}} \frac{2}{\sqrt{\beta}} \frac{ |B(x)|}{|x|}
		\geq
		\frac34 s
	\Big\}
	:=
	D_2(s).
\end{align}
Consequently,
\begin{align}
	\label{eq:Omega1:bd}
	\Pr\big[ \Omega_1 \big]
	\leq
	\Pr\big[ D_1(s) \big] + \Pr\big[ D_2(s) \big ].
\end{align}

Next we turn to bounding $ \Pr[ \Omega_2 ] $.
Consider the last exist $ \tau^{*,s} := \sup \{ x\in [0,\pt_1] : f(x,-s) \geq \frac12\sqrt{s} \} $ 
of $ f(x,-s) $ from the region above $ \frac12 \sqrt{s} $, with the convention $ \sup\emptyset := -\infty $.
Under the occurrence of $ \Omega_2 $, 
set $ (x_1,x_2) = (0\vee \tau^{*,s},\pt_1) $ and $ \lambda=-s $ in~\eqref{eq:Riccati:int} to get
\begin{align*}
	\text{On } \Omega_2 ,
	\quad
	-\frac{\sqrt{s}}2 \ind_\set{\tau^{*,s} \geq 0}
	\geq
	f(x)\Big|_{0\vee \tau^{*,s}}^{\pt_1}
	=
	\int_{0\vee \tau^{*,s}}^{\pt_1} (s-f^2(x)) \d x
	+
	\frac{2}{\sqrt{\beta}} B(x)\Big|_{0\vee \tau^{*,s}}^{\pt_1}.
\end{align*}
Since $ |f(x)| \leq \frac12\sqrt{s} $ for all $ x \in [0\vee \tau^{*,s},\tau_{-,s}] $,
here we have $ \int_{0\vee \tau^{*,s}}^{\pt_1} (s-f^2(x)) \d x \geq \frac34 s (\pt_1-0\vee \tau^{*,s}) $.
This gives
\begin{align}
	\label{eq:Omega=0}
	\text{On } \Omega_2 ,
	\quad
	-\tfrac{\sqrt{s}}2 \ind_\set{\tau^{*,s} \geq 0}
	-
	\tfrac34 s (\pt_1-0\vee \tau^{*,s})
	\geq
	\tfrac{2}{\sqrt{\beta}} B(x)\big|_{0\vee \tau^{*,s}}^{\pt_1}.
\end{align}
Consider further the sub-events
$
	\Omega_{2,\leq}
	:= 
	\{ \tau^{*,s} \leq \pt_{1}-s^{-\frac12} \}\cap\Omega_2
$
and
$
	\Omega_{2,>} 
	:= 
	\{ \tau^{*,s} > \pt_{1}-s^{-\frac12} \}\cap\Omega_2.
$
Under the occurrence of $ \Omega_{2,\leq} $, forgoing the term $ -\frac{\sqrt{s}}2 \ind_\set{\tau^{*,s} \geq 0} $ in~\eqref{eq:Omega=0} gives
\begin{align*}
	\Omega_{2,\leq}
	\subset
	\Big\{
		\sup_{x\in[0,\pt_1-s^{-1/2}]} \frac{2}{\sqrt{\beta}} \frac{ |B(x)-B(\pt_1)|}{|x-\pt_1|}
		\geq
		\frac34 s
	\Big\}
	:=
	\til D_2(s).
\end{align*}
Recall our current assumption $ s \geq t^{\alpha_+\vee(-2\alpha)} $,
which ensures $ s^{-\frac12} \leq \pt_1 $.
Hence under the occurrence of $ \Omega_{2,>} $, we necessarily have $ \tau^{*,s} \geq 0 $.
Forgoing the term $ -\tfrac34 s (\pt_1-0\vee \tau^{*,s}) $ in~\eqref{eq:Omega=0} gives
\begin{align*}
	\Omega_{2,>}
	\subset
	\Big\{
		\inf_{x\in[\pt_1-s^{-1/2},\pt_1]} \frac{2}{\sqrt{\beta}} |B(x)-B(\pt_1)|
		\geq
		\frac{\sqrt{s}}2
	\Big\}
	:=
	\til D_1(s).
\end{align*}
Further, since $ B(\Cdot)-B(\pt_1) \stackrel{\text{law}}{=} B(\Cdot) $,
we have $ \Pr[\til D_1(s)] \leq \Pr[D_1(s)] $ and $ \Pr[\til D_2(s)] \leq \Pr[D_2(s)] $.
The preceding discussion gives
$
	\Pr[ \Omega_2 \big]
	\leq
	\Pr\big[ D_1(s) \big] + \Pr\big[ D_2(s) \big ].
$
Combining this with~\eqref{eq:Omega1:bd} and~\eqref{eq:G1bd:step2} gives
\begin{align}
	\label{eq:G1bd:step2:}
	\Pr\big[ \eigen_{1} < -s  \big]
	\leq
	c\,\Pr\big[ D_1(s) \big] + c\,\Pr\big[ D_2(s) \big ],
	\quad
	s \geq t^{\alpha_+\vee(-2\alpha)}.
\end{align}

\medskip
\noindent\textbf{Step~3: estimating Brownian exist probability.}
We now proceed to bound the r.h.s.\ of~\eqref{eq:G1bd:step2:}.
Referring to the definition~\eqref{eq:D<} of $ D_1(s) $, it is readily checked that $ \Pr[D_1(s)] \leq \exp(-\frac{1}{c}s^\frac32) $.
As for $ D_2(s) $ (defined in~\eqref{eq:D>}),
partition $ [s^{-\frac12},\infty) $ into intervals $ S_k:=[k s^{-\frac12},(k+1)s^{-\frac12}) $, $ k\in\N $ of length $ s^{-\frac12} $.
\begin{align*}
	\Pr\big[ D_2(s) \big]
	\leq
	\sum_{k=1}^\infty 
	\Pr\Big[
		\sup_{x\in S_k} \frac{2}{\sqrt{\beta}} \frac{|B(x)|}{x} \geq \frac{3s}{4}
	\Big]
	\leq
	\sum_{k=1}^\infty 
	\Pr\Big[
		\sup_{x\in[0,(k+1)s^{-1/2}]} |B(x)| \geq \frac{ks^{\frac12}}{c}
	\Big]
	\leq
	\sum_{k=1}^\infty \exp\Big( - \frac{k^2s^\frac32}{c\,(k+1)} \Big).
\end{align*}
The last sum bounded by $ \exp(-\frac{1}{c}s^{-\frac32}) $ for all $ s \geq 1 $.
Consequently,
\begin{align}
	\label{eq:G1bd:step3}
	\Pr\big[ \eigen_{1} \leq -s \big]
	\leq
	\exp\big(-\tfrac{1}{c} s^{\frac32}\big),
	\quad
	s \geq t^{\alpha_+\vee(-2\alpha)}.
\end{align}
Now, write 
\begin{align*}
	G_{i,1}
	=
	\Ex\big[e^{c\,t^\frac13(\kappa+1)(\eigen_{1})_-}\big] 
	= 
	\Pr\big[ (\eigen_{1})_- \geq 0 \big]
	+
	c\,t^{\frac13}(\kappa+1)\int_0^\infty \Pr\big[(\eigen_{1})_-\geq s\big] e^{ct^\frac13(\kappa+1) s}\d s.
\end{align*}
Indeed, $ \Pr[ (\eigen_{1})_- \geq 0 ] =1 $.
For the last integral, bound $ \Pr[(\eigen_{1})_-\geq s] \leq 1 $ for $ s \in [0,t^{\alpha_+\vee(-2\alpha)}] $,
and use the bound~\eqref{eq:G1bd:step3} for $ s > t^{\alpha_+\vee(-2\alpha)} $.
This gives
\begin{align*}
	G_{i,1}
	\leq
	1 
	+
	t^{\frac13+\alpha_+\vee(-2\alpha)}(\kappa+1) e^{c(\kappa+1)t^{\frac13+\alpha_+\vee(-2\alpha)}}
	+
	t^{\frac13}(\kappa+1) e^{c(\kappa+1)^3t}.
\end{align*}
With $ \alpha\in(-\frac13,\frac23) $, the last term $ \exp(c(\kappa+1)^3t) $ dominates for large $ t $.
From this we conclude the desired result: 
$ 	
	\log (G_{i,1})
	\leq
	c (\kappa+1)^3 t,
$
for all $ t\geq 1 $.
\end{proof}

\section{Proof of Theorem~\ref{thm:main:} and Corollary~\ref{cor:goe}}
\label{sect:pf:main}

Passing from Theorem~\ref{thm:main} to Theorem~\ref{thm:main:} and Corollary~\ref{cor:goe} amounts to showing
\begin{lemma}
\label{lem:passing}
Let $ X_t $, $ t>0 $, be a sequence of $ \R $-valued random variables, and let $ b\in(0,\infty) $, $ g \in C[0,\infty) $. If,
for any fixed $ \zeta\in(0,\infty) $ we have
\begin{align}
	\label{eq:passing}
	\lim_{t\to\infty} \tfrac{1}{t^2} \log \big( \Ex\big[ \exp\big(-b e^{X_t+t\zeta}\big) \big]\big) = g(\zeta),
\end{align}
then in fact
\begin{align*}
	\lim_{t\to\infty} \tfrac{1}{t^2} \log \big( \Pr\big[ X_t < -t\zeta \big]\big) = g(\zeta),
\end{align*}
for all fixed $ \zeta\in(0,\infty) $.
\end{lemma}

Indeed, given the identity~\eqref{eq:BG}, 
Theorem~\ref{thm:main:} follows by combining Theorem~\ref{thm:main} for $ (\beta,L)=(2,1) $ and $ g(\zeta)=-\rate(-\zeta) $
and Lemma~\ref{lem:passing} for $ X_t=h(2t,0)+\frac{t}{12} $ and $ b=1 $.
Similarly, given~\eqref{eq:bbcw},
Corollary~\ref{cor:goe} follows by combining Theorem~\ref{thm:main} for $ (\beta,L)=(1,2) $ and $ g(\zeta)=-\frac12 \rate(-\zeta) $
and Lemma~\ref{lem:passing} for$ X_t=h^\hf(2t,0)+\frac{t}{12} $ and $ b=\frac14 $.

\begin{proof}[Proof of Lemma~\ref{lem:passing}]
Write $ F(x) := \exp(-be^x) $ for the double exponential function.
Fix $ \delta\in (0,\zeta) $.
We indeed have $ F(x+\delta t) \leq \ind_\set{x<0} + \exp(-be^{\delta t})  $
and $ F(x-\delta t) \geq \exp(-be^{-\delta t}) \ind_\set{x<0} $.
From this we conclude
\begin{align}
	\label{eq:main:2}
	\Pr\big[ X_t<-t\zeta \big] + \exp(-be^{\delta t})	
	&\geq
	\Ex\big[ F(X_t+t(\zeta+a)) \big],
\\
	\label{eq:main:3}
	e^{-be^{-\delta t}} \Pr\big[ X_t<-t\zeta \big]
	&\leq
	\Ex\big[ F(X_t+t(\zeta-a)) \big].
\end{align}
Combining the given assumption~\eqref{eq:passing} for $ \zeta\mapsto \zeta+a $ with~\eqref{eq:main:2} gives,
for all large enough $ t $,
\begin{align*}
	\Pr\big[ X_t<-t\zeta \big] 	
	>
	\tfrac12 \exp( t^2g(\zeta+a))
	-\exp(-be^{\delta t}).
\end{align*} 
On the r.h.s., the first term dominates as $ t\to\infty $ (\emph{regardless} of the sign of $ g(\zeta+a) $).
Consequently, for all large enough $ t $,
\begin{align}
	\label{eq:main:4}
	\Pr\big[ X_t<-t\zeta \big] 	
	>
	\tfrac12 \exp( t^2 g(\zeta+a) )
	=
	\tfrac12 \exp( - \tfrac{t^2}{c(\zeta,a,g)} ).	
\end{align} 

Take logarithm on both sides of~\eqref{eq:main:2}--\eqref{eq:main:3}, and divide the result by $ t^2 $.
With the aid of the inequality $ \log(x_1+x_2) \leq \log x_1 + \frac{x_2}{x_1} $, valid for $ x_1,x_2\in (0,1] $,
upon sending $ t\to\infty $ we have
\begin{align}
	\label{eq:main:5}
	\liminf_{t\to\infty} \frac{1}{t^2} \log \big(  \Pr\big[ X_t < -\zeta t \big] \big)
	+
	\limsup_{t\to\infty} \frac{e^{-be^{\delta t}}}{t^2 \Pr[ X_t < -\zeta t ] }	
	&\geq
	g(\zeta+\delta),
\\
	\notag
	\limsup_{t\to\infty} \frac{1}{t^2} \log \big(  \Pr\big[ X_t < -\zeta t \big] \big)
	&\leq
	g(\zeta-\delta).
\end{align}
On the l.h.s.\ of~\eqref{eq:main:5}, use the bound~\eqref{eq:main:4}, we see that the second term is in fact zero.
Further taking $ \delta\downarrow 0 $ and using the continuity of $ g $,
we conclude the desired result.
\end{proof}

\bibliographystyle{alphaabbr}
\bibliography{kpzLD}
\end{document}